\documentclass[oneside,english]{amsart}
\usepackage[T1]{fontenc}
\usepackage[latin9]{inputenc}
\usepackage{geometry}
\usepackage{amsbsy}
\usepackage{amstext}
\usepackage{amsthm}
\usepackage{amssymb}

\makeatletter

\newcommand{\noun}[1]{\textsc{#1}}

\numberwithin{equation}{section}
\numberwithin{figure}{section}
\theoremstyle{plain}
\newtheorem{thm}{\protect\theoremname}
\theoremstyle{definition}
\newtheorem{defn}[thm]{\protect\definitionname}
\theoremstyle{remark}
\newtheorem{rem}[thm]{\protect\remarkname}
\theoremstyle{plain}
\newtheorem{prop}[thm]{\protect\propositionname}
\theoremstyle{plain}
\newtheorem{cor}[thm]{\protect\corollaryname}
\theoremstyle{definition}
\newtheorem{example}[thm]{\protect\examplename}
\theoremstyle{plain}
\newtheorem{lem}[thm]{\protect\lemmaname}

\makeatother

\usepackage{babel}
\providecommand{\corollaryname}{Corollary}
\providecommand{\definitionname}{Definition}
\providecommand{\examplename}{Example}
\providecommand{\lemmaname}{Lemma}
\providecommand{\propositionname}{Proposition}
\providecommand{\remarkname}{Remark}
\providecommand{\theoremname}{Theorem}

\begin{document}
\title{Regular orders for triangular fuzzy numbers\\
and the Weak Law of Trichotomy}
\author{Jaime Cesar dos Santos Filho }
\email{jaime.cesar@ufrpe.br}
\address{Unidade Acadêmica do Cabo de Santo Agostinho \\
Universidade Federal Rural de Pernambuco\\
Cabo de Santo Agostinho, Pernambuco, 54518-430, Brazil}
\keywords{Triangular fuzzy numbers; Regular orders; Law of trichotomy; Fuzzy
absolute value and distance; \emph{MIN-MAX }compatibility; Arithmetic
compatibility; }
\begin{abstract}
Building upon specific compatibility conditions, we establish fundamental
structural results concerning ordering relations for triangular fuzzy
numbers. We demonstrate that orders satisfying compatibility with
arithmetic operations, \emph{MIN-MAX} operators, and the Weak Law
of Trichotomy (WLT) are completely determined on the fibers of the
natural projection to real numbers. Furthermore, such orders naturally
induce - in analogy with real numbers - well-defined notions of fuzzy
absolute value and fuzzy distance that preserve the essential properties
of their classical counterparts. These results enable us to characterize
open and closed balls through interval representations, providing
a robust theoretical framework for future studies regarding metric
properties of fuzzy numbers.
\end{abstract}

\maketitle

\section{Introduction}

Among various types of fuzzy quantities, fuzzy numbers occupy a central
position in theoretical developments, as they capture intuitive notions
of approximate quantities such as \textquotedbl numbers close to
$1$\textquotedbl , \textquotedbl approximately $0$\textquotedbl ,
or \textquotedbl numbers around the interval $\left[1,2\right]$\textquotedbl .
Within this class, triangular fuzzy numbers have yielded particularly
robust results regarding orders satisfying desirable conditions, owing
to their mathematical simplicity. The set of triangular fuzzy numbers,
denoted by $\mathbb{T}$, can be naturally identified with ordered
triples 
\[
\mathbb{T}=\left\{ \left(a_{1},a,a_{2}\right)\in\mathbb{R}^{3}\,:\,a_{1}\leq a\leq a_{2}\right\} ,
\]
 as detailed in \emph{Section  \ref{sec:DefinicoesBasicas}}.

The ordering of fuzzy quantities, particularly fuzzy numbers, represents
a fundamental challenge that has attracted considerable research attention.
While numerous ranking methods and preference criteria have been proposed,
significant disagreement persists in the field, with increasingly
technical and complex approaches emerging.

Our work adopts a novel perspective by beginning with desired conditions
that an order on $\mathbb{T}$ should satisfy, then systematically
deriving the properties of orders compatible with these conditions.
Unconventional concepts mentioned here are formally defined in \emph{Section
\ref{sec:DefinicoesBasicas}}.

Our first condition concerns \emph{compatibility with arithmetic operations}
on $\mathbb{T}$ - a property commonly required in fuzzy number classification
to enable feasible algebraic manipulation of inequalities. \emph{Theorem}
\emph{\ref{th:OrdemNosAnuladores}} establishes that arithmetic-compatible
orders are completely determined within each \emph{nullifying set}
(\ref{eq:Anuladores}).

The \emph{Weak Law of Trichotomy} (WLT) represents the closest analogue
to the trichotomy law for ordered fields that can be achieved in $\mathbb{T}$.
\emph{Theorem} \emph{\ref{thm:LFTeValorAbsoluto}} demonstrates that
orders satisfying both arithmetic compatibility and WLT induce:
\begin{enumerate}
\item A \emph{fuzzy absolute value} extending the classical real-valued
absolute value while preserving its essential properties.
\item An associated \emph{fuzzy distance} extending the standard metric
on $\mathbb{R}$.
\end{enumerate}
These results enable the characterization of \emph{open and closed
balls} through \emph{intervals} in $\mathbb{T}$ (\emph{Theorems}
\emph{\ref{thm:BolaseIntervalos1} }and \emph{\ref{thm:BolasEIntervalos2}}),
mirroring classical real analysis.

\emph{Section} \emph{\ref{sec:MIN-MAX-compatibilidadeEOrdemSomaSuperior}}
introduces \emph{compatibility with MIN-MAX operators} as an additional
condition. Since these operators derive from the \emph{extension principle}
(like arithmetic operations), this represents a natural requirement
for fuzzy number ranking methods. Our main structural result shows
that orders satisfying all three conditions (arithmetic compatibility,
\emph{MIN-MAX} compatibility, and WLT) are completely determined on
\emph{fibers of the natural projection}:
\begin{thm}
If $\preceq$ is an order on $\mathbb{T}$ compatible with arithmetic
operations, with MIN-MAX operators and satisfies WLT then for each
$t\in\mathbb{R}$:\\
i) $\left(x_{1},t,y_{1}\right)\preceq\left(x_{2},t,y_{2}\right)\Longleftrightarrow\left\{ \begin{array}{l}
\textrm{\ensuremath{x_{1}+y_{1}<x_{2}+y_{2},} or }\\
\textrm{\ensuremath{x_{1}+y_{1}=x_{2}+y_{2}} and \ensuremath{y_{1}\leq y_{2}},}
\end{array}\right.$ if $I_{0}\subseteq P_{\preceq}$. Or \\
ii) $\left(x_{1},t,y_{1}\right)\preceq\left(x_{2},t,y_{2}\right)\Longleftrightarrow\left\{ \begin{array}{l}
\textrm{\ensuremath{x_{1}+y_{1}<x_{2}+y_{2},} or }\\
\textrm{\ensuremath{x_{1}+y_{1}=x_{2}+y_{2}} and \ensuremath{x_{1}\leq x_{2}},}
\end{array}\right.$ if $I_{0}\cap P_{\preceq}=\emptyset$.
\end{thm}

In this work, an order compatible with arithmetic operations and \emph{MIN-MAX}
operators is called a \emph{regular} order. Regular orders satisfying
the WLT differ primarily in their treatment of triangular fuzzy numbers
from different projection fibers. The simplest way to \textquotedblleft accommodate\textquotedblright{}
distinct fibers is given by \emph{compatibility with the natural projection}.
The most restrictive condition we consider - compatibility with natural
projection - arises naturally in certain ranking methods (e.g., Molinari
\cite{Molinari}, see \emph{Example} \emph{\ref{exa:OrdemMolinari}}).
\emph{Theorem} \emph{\ref{thm:UnicidadeOrdemNormalLFT}} establishes
that only two regular orders satisfy both WLT and projection compatibility:
the \emph{upper sum} \emph{(\ref{eq:SomaSuperior})} and \emph{lower
sum} \emph{(\ref{eq:SomaInferior})}.

\emph{Section} \emph{\ref{sec:ExemplosNumericos}} analyzes various
fuzzy number ranking methods through numerical examples, examining
constraints on their extension to total orders on $\mathbb{T}$. We
compare these results with the \emph{total sum} \emph{(\ref{eq:OrdemSomaTotal})}
and \emph{upper sum} \emph{(\ref{eq:SomaSuperior})} orders - both
regular orders satisfying WLT.

\section{\label{sec:DefinicoesBasicas} Basic definitions and compatibility
conditions}

In this section, we present fundamental definitions concerning methods
for ranking triangular fuzzy numbers, the compatibility conditions
developed in this work, as well as the concepts of fuzzy absolute
value and fuzzy distance.

In this work, $\mathbb{R}$ denotes the set of \emph{real numbers},
$\mathbb{T}$ the set of \emph{triangular fuzzy numbers} and $\mathcal{F}\left(\mathbb{R}\right)$
the set of\emph{ fuzzy numbers}. We adopt the characterization of
fuzzy numbers given by Klir-Yuan \cite{Klir-Yuan}: a fuzzy number
is a function $\alpha:\mathbb{R}\longrightarrow\left[0,1\right]$
given by

\[
\alpha\left(t\right)=\left\{ \begin{array}{cl}
l\left(t\right), & t\in\left(-\infty,a\right);\\
1, & t\in\left[a,b\right];\\
r\left(t\right), & t\in\left(b,\infty\right),
\end{array}\right.
\]
 where $\left[a,b\right]\neq\emptyset$ is a closed interval, $l:\left(-\infty,a\right)\longrightarrow\left[0,1\right]$
is monotonic increasing, right-continuous and there exists $w_{1}\leq a$
with $l\left(t\right)=0$ for $t\in\left(-\infty,w_{1}\right)$, and
$r:\left(b,\infty\right)\longrightarrow\left[0,1\right]$ is monotonic
decreasing, left-continuous such that there exists $w_{2}\geq b$
with $r\left(t\right)=0$ for all $t\in\left(w_{2},\infty\right)$.
Here, $w_{1}=inf\,S_{upp}\left(\alpha\right)$ and $w_{2}=sup\,S_{upp}\left(\alpha\right)$
denote the \emph{infimum} and \emph{supremum}, respectively, of the
\emph{support} of $\alpha$.

Among the different types of fuzzy numbers, \emph{triangular fuzzy
numbers} are the most commonly used in examples and applications due
to their simplicity. A triangular fuzzy number is a fuzzy number $\alpha$
given by

\[
\alpha\left(t\right)=\left\{ \begin{array}{cl}
\frac{t-a_{1}}{a-a_{1}}, & t\in\left[a_{1},a\right);\\
1, & t=a;\\
\frac{a_{2}-t}{a_{2}-a}, & t\in\left(a,a_{2}\right];\\
0, & \textrm{otherwise,}
\end{array}\right.
\]
 for real numbers $a_{1}\leq a\leq a_{2}$, where $a_{1}=inf\,S_{upp}\left(\alpha\right)$
and $a_{2}=sup\,S_{upp}\left(\alpha\right)$. Thus, each triangular
fuzzy number $\alpha$ is uniquely determined by a triple of real
numbers $\left(a_{1},a,a_{2}\right)$, where $a_{1}\leq a\leq a_{2}$,
and we write $\alpha=\left(a_{1},a,a_{2}\right)$. Consequently, the
set of triangular fuzzy numbers is naturally identified with the set
of ordered triples
\[
\mathbb{T}=\left\{ \left(a_{1},a,a_{2}\right)\in\mathbb{R}^{3}\,:\,a_{1}\leq a\leq a_{2}\right\} 
\]
 (for further references on this approach, see \cite{Al-Amleh,Molinari,Zumelzu}).
Real numbers are identified as a subset of $\mathbb{T}$, $\mathbb{R}\simeq\left\{ \left(t,t,t\right)\,:\,t\in\mathbb{R}\right\} $,
referred to as \emph{scalars}. Here, no distinction is made between
$t\in\mathbb{R}$ and $\left(t,t,t\right)\in\mathbb{T}$, and we treat
$\mathbb{R}$ as a subset of $\mathbb{T}$. In terms of notation,
lowercase Greek letters ($\alpha,\beta,\gamma,\dots$) denote for
fuzzy numbers, while standard lowercase letters ($a,b,c,\dots$) denote
real numbers, unless stated otherwise. The symbols $\leq$, $<$ ,
and $\left|\cdot\right|$ are reserved for the usual order and absolute
value in $\mathbb{R}$.

\subsection{Ranking methods, arithmetic operations and the weak law of trichotomy}

~

Based on various ranking methods for fuzzy quantities, Wang-Kerre
\cite{Wang and Kerre} proposed a set of reasonable properties that
a fuzzy quantity ranking method should satisfy. According to properties
$\boldsymbol{A}_{1}-\boldsymbol{A}_{7}$ proposed in \emph{Section
3} of \cite{Wang and Kerre}, we define:
\begin{defn}
\label{def:MetodoRazoavel} A \emph{reasonable} \emph{ranking method}
$\preceq$ over $\mathbb{T}$ is a relation satisfying: \\
\emph{i)} \emph{(Reflexivity)} for all $\alpha\in\mathbb{T}$, $\alpha\preceq\alpha$;\\
\emph{ii) (Antisymmetry up to equivalence)} given $\alpha,\beta\in\mathbb{T}$,
if $\alpha\preceq\beta$ and $\beta\preceq\alpha$ then $\beta\sim\alpha$
(i.e., these elements are equivalent under this ranking method);\\
\emph{iii)} \emph{(Transitivity)} given $\alpha,\beta,\gamma\in\mathbb{T}$,
if $\alpha\preceq\beta$ and $\beta\preceq\gamma$ then $\alpha\preceq\gamma$;\\
\emph{iv)} \emph{(Sum compatibility)} given $\alpha,\beta,\gamma\in\mathbb{T}$,
if $\alpha\preceq\beta$ then $\beta+\gamma\preceq\alpha+\gamma$;\\
\emph{v)} \emph{(Scalar multiplication compatibility)} given $\alpha,\beta\in\mathbb{T}$
and $t\in\mathbb{R}$ with $0\preceq t$, if $\alpha\preceq\beta$
then $t\alpha\preceq t\beta$;\\
\emph{vi) (Strict order for disjoint supports)} given $\alpha,\beta\in\mathbb{T}$,
if $sup\,S_{upp}\left(\alpha\right)<inf\,S_{upp}\left(\beta\right)$
then $\alpha\prec\beta$ (i.e., $\alpha\preceq\beta$ and $\beta\npreceq\alpha$). 
\end{defn}

\begin{rem}
Due to property \emph{(vi)}, any reasonable ranking method $\preceq$
extends the usual ordering of real numbers: if $r<s$ in $\mathbb{R}$
then $r\prec s$ in $\mathbb{T}$.\emph{ }
\end{rem}

\begin{rem}
Condition \emph{(v)} requires $0\preceq t$. Under \emph{(vi)}, this
is equivalent to $0\leq t$.
\end{rem}

Our primary interest lies in total orders over $\mathbb{T}$. A total
order (or simply an order) is a reflexive, antisymmetric, transitive
relation over $\mathbb{T}$ in which any two elements are comparable:
\begin{defn}
\label{def:OrdemTotal} A \emph{(total) order} $\preceq$ on $\mathbb{T}$
is a relation satisfying:\\
\emph{i)} \emph{(Totality)} given $\alpha,\beta\in\mathbb{T}$, $\alpha\preceq\beta$
or $\beta\preceq\alpha$;\noun{ }\\
\emph{ii)} \emph{(Reflexivity)} for all $\alpha\in\mathbb{T}$, $\alpha\preceq\alpha$;\\
\emph{iii)} \emph{(Antisymmetry)} given $\alpha,\beta\in\mathbb{T}$,
if $\alpha\preceq\beta$ and $\beta\preceq\alpha$, then $\alpha=\beta$;\\
\emph{iv)} \emph{(Transitivity)} given $\alpha,\beta,\gamma\in\mathbb{T}$,
if $\alpha\preceq\beta$ and $\beta\preceq\gamma$, then $\alpha\preceq\gamma$.
\end{defn}

Note that totality implies reflexivity. A relation satisfying \emph{(ii
- iv)} without totality is called a \emph{partial order}. A reflexive
and transitive relation that is not necessarily antisymmetric is called
a \emph{preorder} (or \emph{preference relation}), which may be total
or partial. Here, an order refers to a total, antisymmetric preorder. 

If $\preceq$ is a preorder (total or partial, antisymmetric or not),
its dual preorder, denoted $\preceq^{*}$, is defined by $\alpha\preceq^{*}\beta$
if and only if $\alpha\succeq\beta$.

Using the \emph{extension principle} (see Klir-Yuan \cite{Klir-Yuan}),
the standard arithmetic operations on $\mathbb{R}$ ---addition,
multiplication, and the $max$ and $min$ operators--- extend to
fuzzy numbers as follows: given $\alpha,\beta\in\mathcal{F}\left(\mathbb{R}\right)$,\\
\emph{i)} $\left(\alpha\pm\beta\right)\left(z\right)=\underset{z=x\pm y}{sup}\,min\left\{ \alpha\left(x\right),\beta\left(y\right)\right\} $;
and\\
\emph{ii)} $\left(\alpha\cdot\beta\right)\left(z\right)=\underset{z=x\cdot y}{sup}\,min\left\{ \alpha\left(x\right),\beta\left(y\right)\right\} $;
and\\
\emph{iii)} $\left(\alpha\div\beta\right)\left(z\right)=\underset{z=x\div y}{sup}\,min\left\{ \alpha\left(x\right),\beta\left(y\right)\right\} $,
case $0\notin S_{upp}\left(\beta\right)$;\\
\emph{iv) $MIN\left(\alpha,\beta\right)\left(z\right)=\underset{z=min\left\{ x,y\right\} }{sup}\,min\left\{ \alpha\left(x\right),\beta\left(y\right)\right\} $
}and $MAX\left(\alpha,\beta\right)\left(z\right)=\underset{z=max\left\{ x,y\right\} }{sup}\,min\left\{ \alpha\left(x\right),\beta\left(y\right)\right\} $.

~

This approach extends real arithmetic to $\mathcal{F}\left(\mathbb{R}\right)$
while preserving several desirable properties (see \cite{Klir-Yuan}).
However, multiplication and division are not well-defined on $\mathbb{T}$
unless one operand is a scalar ($\beta\neq0$ in the case of division).
Restricting these operations to triangular fuzzy numbers yields the
following definitions:
\begin{defn}
\label{def:OperacoesAritmeticas} If $\alpha=\left(a_{1},a,a_{2}\right)$
and $\beta=\left(b_{1},b,b_{2}\right)$ are triangular fuzzy numbers
and $t=\left(t,t,t\right)$ is a scalar,\\
\emph{i) (Sum)} $\alpha+\beta=\left(a_{1}+b_{1},a+b,a_{2}+b_{2}\right);$\\
\emph{ii) (Scalar multiplication)} $t\alpha=\begin{cases}
\left(ta_{1},ta,ta_{2}\right), & t\geq0\\
\left(ta_{2},ta,ta_{1}\right), & t<0
\end{cases}$.

Subtraction is defined as $\alpha-\beta=\alpha+\left(-1\right)\beta$.
\end{defn}

A key difference from real numbers is the existence of nonzero $0$-symmetric
numbers: $\alpha\in\mathbb{T}$ such that $\alpha=-\alpha\neq0$.
We denote the set of such numbers by
\begin{equation}
I_{0}=\left\{ \alpha\in\mathbb{T}\backslash\left\{ 0\right\} \,:\,\alpha=-\alpha\right\} =\left\{ \left(-t,0,t\right)\in\mathbb{T}\,:\,t>0\in\mathbb{R}\right\} .\label{eq:0SimetricosNaoNulos}
\end{equation}
 The existence of these numbers prevents the law of trichotomy ---\textquotedbl given
$\alpha\in\mathbb{T}$, exactly one of the following holds: $\alpha=0$,
$\alpha\in P$ or $-\alpha\in P$\textquotedbl --- from being satisfied
for any $P\subset\mathbb{T}$. The closest feasible condition is the
\emph{weak law of trichotomy} (WLT):
\begin{defn}
\label{def:LFT} An order $\preceq$ on $\mathbb{T}$ satisfies the
\emph{weak law of trichotomy (WLT)} if, for all $\alpha\notin I_{0}$,
exactly one of the following holds: $\alpha=0$, $0\prec\alpha$ or
$0\prec-\alpha$.
\end{defn}

Note that the WLT is preserved under duality.
\begin{defn}
For $\alpha=\left(a_{1},a,a_{2}\right)\in\mathbb{T}$, we define the\emph{
nullifying set} of $\alpha$ to be
\begin{equation}
N_{ull}\left(\alpha\right)=\left\{ \beta\in\mathbb{T}\,:\,\alpha-\beta\in I_{0}\cup\left\{ 0\right\} \right\} =\left\{ \left(x,a,y\right)\in\mathbb{T}\,:\,x+y=a_{1}+a_{2}\right\} .\label{eq:Anuladores}
\end{equation}
 
\end{defn}

The following properties can be derived from (\ref{eq:Anuladores}):
\begin{enumerate}
\item $N_{ull}\left(\alpha\right)\cap N_{ull}\left(\beta\right)\neq\emptyset$
$\Longleftrightarrow N_{ull}\left(\alpha\right)=N_{ull}\left(\beta\right)$$\Longleftrightarrow$$\beta\in N_{ull}\left(\alpha\right)$;
\item $N_{ull}\left(0\right)=I_{0}\cup\left\{ 0\right\} $; 
\item $\alpha+N_{ull}\left(0\right)=\left\{ \left(b_{1},a,b_{2}\right)\in N_{ull}\left(\alpha\right)\,:\,a_{2}\le b_{2}\right\} $;
\item $N_{ull}\left(\alpha+t\beta\right)=N_{ull}\left(\alpha\right)+tN_{ull}\left(\beta\right)$. 
\end{enumerate}

\subsection{Compatibility conditions}

~
\begin{defn}
\emph{\label{def:CompatibilidadeOp.Aritm.}} \emph{(Compatibility
with arithmetic operations)} A order $\preceq$ over $\mathbb{T}$
is said to be:\\
\emph{ i)} \emph{compatible with sum:} for $\alpha,\beta,\gamma\in\mathbb{T}$,
if $\alpha\preceq\beta$ then $\alpha+\gamma\preceq\beta+\gamma$;\\
\emph{ii)} \emph{compatible with scalar multiplication:} for $\alpha,\beta\in\mathbb{T}$,
if $\alpha\preceq\beta$ then $t\alpha\preceq t\beta$ for every scalar
$t\geq0$.

If an order satisfies both conditions, it is said to be \emph{compatible
with arithmetic operations} or \emph{arithmetic-compatible}.
\end{defn}

Note: if $\preceq$ is total and compatible with sum, then $\alpha\prec\beta$
implies $\alpha+\gamma\prec\beta+\gamma$, since $\alpha+\gamma=\beta+\gamma$
would imply $\alpha=\beta$. More generally, the the cancellation
law holds: if $\alpha+\gamma\preceq\beta+\gamma$ then $\alpha\preceq\beta$.

\emph{Theorem} \emph{\ref{th:OrdemNosAnuladores}} shows that if $\preceq$
is compatible with arithmetic operations, then it is fully determined
within each nullifying set (\ref{eq:Anuladores}). Furthermore, under
$\preceq$, each nullifying set has either a smallest or largest element
(as seen in \emph{Theorem }\ref{th:OrdemNosAnuladores}).
\begin{defn}
\emph{\label{def:CompatibilidadeMinMax}} \emph{(Compatibility with
MIN-MAX operators)} An order $\preceq$ on $\mathbb{T}$ is \emph{compatible
with MIN-MAX operators} if, given $\alpha,\beta\in\mathbb{T}$ with
$MIN\left(\alpha,\beta\right)=\alpha$, then $\alpha\preceq\beta$.
\end{defn}

The next result is well known and concerns the \emph{MIN-MAX} operators
on triangular fuzzy numbers. Its proof will be omitted.
\begin{prop}
\label{prop:TFNcomparaveis} Given $\alpha=\left(a_{1},a,a_{2}\right)$
and $\beta=\left(b_{1},b,b_{2}\right)$ triangular fuzzy numbers with
$a\leq b$, the following cases arise:\\
i) if $a_{1}\leq b_{1}$ and $a_{2}\leq b_{2}$, then $MIN\left(\alpha,\beta\right)=\alpha$
and $MAX\left(\alpha,\beta\right)=\beta$; or\\
ii) if $a=b$ and $a_{1}<b_{1}<b_{2}<a_{2},$ then $MIN\left(\alpha,\beta\right),MAX\left(\alpha,\beta\right)\in\mathbb{T}$
despite $MIN\left(\alpha,\beta\right),MAX\left(\alpha,\beta\right)\notin\left\{ \alpha,\beta\right\} $;\\
iii) otherwise, $MIN\left(\alpha,\beta\right),MAX\left(\alpha,\beta\right)\notin\mathbb{T}$.

In particular, $MIN\left(\alpha,\beta\right)=\alpha$ if and only
if $a_{1}\leq b_{1}$, $a\leq b$ and $a_{2}\leq b_{2}$.
\end{prop}

By the above result, \emph{MIN-MAX} operators induce a partial order
in $\mathbb{T}$, known as \emph{Klir-Yuan order}:
\begin{align*}
\alpha & \leq_{KY}\beta\,\Longleftrightarrow\,MIN\left(\alpha,\beta\right)=\alpha\,\Longleftrightarrow\,MAX\left(\alpha,\beta\right)=\beta\\
 & \Longleftrightarrow\,\textrm{\ensuremath{inf\,S_{upp}\left(\alpha\right)\leq inf\,S_{upp}\left(\beta\right)} and \ensuremath{\pi\left(\alpha\right)\leq\pi\left(\beta\right)} and \ensuremath{sup\,S_{upp}\left(\alpha\right)\leq sup\,S_{upp}\left(\beta\right)}. }
\end{align*}

In \cite{Zumelzu}, Zumelzu et al. define admissible orders on $\mathcal{F}\left(\mathbb{R}\right)$
as those that are total and extend $\leq_{KY}$. Recently, García-Zamora
et al. \cite{Garcia et al} worked on the concept of \emph{OWA operators}
for fuzzy numbers with respect to admissible orders.

Since compatibility with arithmetic operations and\emph{ MIN-MAX}
operators is a common requirement, we unify these notions:
\begin{defn}
\label{def:OrdemRegular} A \emph{regular order} on $\mathbb{T}$
is an order compatible with both arithmetic operations and \emph{MIN-MAX}
operators.
\end{defn}

From \emph{Definition} \emph{\ref{def:MetodoRazoavel}} and \emph{Proposition}
\emph{\ref{prop:TFNcomparaveis}}, every regular order is a reasonable
ranking method.

Triangular fuzzy numbers often model \textquotedbl values close to
a real number\textquotedbl , where $\alpha=\left(t_{1},t,t_{2}\right)$
has $t$ as the expected/probable value and $t_{2}-t$, $t-t_{1}$
as uncertainty margins (upper and lower margins, respectively).
\begin{defn}
\label{def:ProjNatural} The \emph{natural projection} is the map
$\pi:\mathbb{T}\longrightarrow\mathbb{R}$, where $\pi\left(t_{1},t,t_{2}\right)=t$.
The preimage $\pi^{-1}\left(t\right)$ is called the \emph{fiber}
over $t\in\mathbb{R}$. 
\end{defn}

The pullback of the usual order of the real numbers via $\pi$ defines
a total preorder on $\mathbb{T}$:
\[
\alpha\preceq_{\pi}\beta\,\Longleftrightarrow\,\pi\left(\alpha\right)\leq\pi\left(\beta\right).
\]
The equivalence classes of $\preceq_{\pi}$ are precisely the fibers
of $\pi$.

\emph{Theorem \ref{thm:OrdensCompativeisSobreAsFibras}} states that
regular orders satisfying the weak law of trichotomy (WLT) are fully
determined on each fiber $\pi^{-1}\left(t\right)$, depending on whether
they admit positive $0$-symmetric elements.
\begin{defn}
\label{def:CompProjNatural} An order $\preceq$ on $\mathbb{T}$
is \emph{compatible with the natural projection} if it extends $\preceq_{\pi}$
( i.e. if $\alpha\prec_{\pi}\beta$, then $\alpha\prec\beta$).
\end{defn}

Such an order preserves fiber structure: no element \textquotedbl close\textquotedbl{}
to $a$ lies between two elements \textquotedbl close\textquotedbl{}
to $b$ if $a\neq b$. \emph{Theorem \ref{thm:UnicidadeOrdemNormalLFT}}
shows that the \emph{upper-sum} and \emph{lower-sum} orders (\emph{Eqs}.
\ref{eq:SomaSuperior}, \ref{eq:SomaInferior}) are the only regular,
WLT-satisfying orders compatible with $\pi$.

\subsection{Absolute value and fuzzy distance}

~

The absolute value in $\mathbb{R}$ is $\left|x\right|=max\left\{ x,-x\right\} $,
inducing the distance $\left|x-y\right|$. For a total order $\preceq$
on $\mathbb{T}$, the fuzzy absolute value is: 
\[
\left|\alpha\right|_{\preceq}=\underset{\preceq}{max}\left\{ \alpha,-\alpha\right\} .
\]

However, $\left|\cdot\right|_{\preceq}$ may not behave like the real
absolute value (e.g., $\left|\alpha\right|_{\preceq}\prec0$ is possible).
A well-behaved fuzzy absolute value should satisfy for arbitrary elements:\\
\emph{i)} $\left|\alpha\right|_{\preceq}\succeq0$; \\
\emph{ii)} $\left|t\alpha\right|_{\preceq}=\left|t\right|\left|\alpha\right|_{\preceq}$
(where $\left|t\right|$ is the real absolute value);\\
\emph{iii)} $\left|\alpha+\beta\right|_{\preceq}\preceq\left|\alpha\right|_{\preceq}+\left|\beta\right|_{\preceq}$;\\
\emph{iv) $\left|\alpha-\gamma\right|_{\preceq}\preceq\left|\alpha-\beta\right|_{\preceq}+\left|\beta-\gamma\right|_{\preceq}$.}
\begin{defn}
\label{def:ValorAbsolutoDifuso} An order $\preceq$ induces a \emph{fuzzy
absolute value} if $\left|\cdot\right|_{\preceq}$ satisfies \emph{(i)-(iv)}
above. In this case, the \emph{fuzzy distance} induced by $\preceq$
is:
\[
D_{\preceq}\left(\alpha,\beta\right)=\left|\alpha-\beta\right|_{\preceq},
\]
which satisfies:\\
\emph{i) (Positivity)} $0\preceq D_{\preceq}\left(\alpha,\beta\right)$,
and $D_{\preceq}\left(\alpha,\beta\right)=0$ if and only if $\alpha=\beta\in$;
\\
\emph{ii) (Symmetry)} $D_{\preceq}\left(\alpha,\beta\right)=D_{\preceq}\left(\beta,\alpha\right)$;\\
\emph{iii) (Triangular inequality)} $D_{\preceq}\left(\alpha,\gamma\right)\preceq D_{\preceq}\left(\alpha,\beta\right)+D_{\preceq}\left(\beta,\gamma\right)$.
\end{defn}

\begin{defn}
\emph{(Open/closed balls)} For $\beta,\gamma\in\mathbb{T}$ with $\gamma\succeq0$: 

$\bullet$ the \emph{open ball centered at $\beta$ of radius $\gamma$}
is: 
\[
\mathcal{B}_{\preceq}\left(\beta,\gamma\right)=\left\{ \alpha\in\mathbb{T}\,:\,\left|\alpha-\beta\right|_{\preceq}\prec\gamma\right\} .
\]

$\bullet$ the \emph{closed ball} is: 
\[
\overline{\mathcal{B}_{\preceq}}\left(\beta,\gamma\right)=\left\{ \alpha\in\mathbb{T}\,:\,\left|\alpha-\beta\right|_{\preceq}\preceq\gamma\right\} .
\]
\end{defn}

\begin{defn}
\emph{(Intervals) }A subset $I\subseteq\mathbb{T}$ is an \emph{interval}
if for all $\alpha,\beta\in I$ with $\alpha\prec\beta$, every $\gamma$
satisfying $\alpha\prec\gamma\prec\beta$ is in $I$. 
\end{defn}

For $\alpha,\beta\in\mathbb{T}$, we denote:

$\bullet$ \emph{open interval:} $\left(\alpha,\beta\right)_{\preceq}=\left\{ \gamma\in\mathbb{T}\,:\,\alpha\prec\gamma\prec\beta\right\} $;

$\bullet$ \emph{closed interval} $\left[\alpha,\beta\right]_{\preceq}=\left\{ \gamma\in\mathbb{T}\,:\,\alpha\preceq\gamma\preceq\beta\right\} $;

$\bullet$ Analogously for the \emph{half-open intervals} $\left[\alpha,\beta\right)_{\preceq}$
and $\left(\alpha,\beta\right]_{\preceq}$.
\begin{rem}
These are set-theoretic definitions (not topological conditions).
\end{rem}

\emph{Corollary \ref{cor:AnuladoresSaoIntervalos}} establishes that
if $\preceq$ is compatible with the arithmetic operations and satisfies
the weak law of trichotomy then each nullifying set (\ref{eq:Anuladores})
is an interval, as are the non-zero $0$-symmetric set $I_{0}$. \emph{Theorem}
\emph{\ref{thm:LFTeValorAbsoluto}} shows that these same orders,
when they have positive $0$-symmetric numbers, induce the best possible
notion of absolute value on the triangular fuzzy numbers. By adding
compatibility with \emph{MIN-MAX} operators, we can characterize the
open/closed balls through intervals. This is done in \emph{Theorems}
\emph{\ref{thm:BolaseIntervalos1} }and \emph{\ref{thm:BolasEIntervalos2}}.

\section{\label{sec:OrdensCompativeisLeiFracaTricotomia} Compatibility with
arithmetic operations, weak law of trichotomy \protect \\
and fuzzy absolute value}

Compatibility with arithmetic operations determines the order on each
nullifying set, depending solely on whether the order admits positive
$0$-symmetric numbers. Specifically, under such an order, every nullifying
set possesses either a minimum or maximum element. These results follow
directly from \emph{Theorem} \emph{\ref{th:OrdemNosAnuladores}}. 

\emph{Corollary} \emph{\ref{cor:AnuladoresSaoIntervalos}} demonstrates
that the weak law of trichotomy (WLT), when combined with compatibility
to arithmetic operations, implies that each nullifying set ---as
well as the set $I_{0}$ of nonzero $0$-symmetric numbers--- forms
an interval.

Furthermore, \emph{Theorem} \emph{\ref{thm:Compat.Op.AritmeticaseLFTePositivos}}
establishes that if two arithmetic-compatible orders share the same
set of positive elements and at least one satisfies the WLT, then
these orders must coincide.

When an arithmetic-compatible order satisfies the WLT and admits positive
$0$-symmetric numbers, it induces a well-defined fuzzy absolute value
and fuzzy distance, as proven in \emph{Theorem} \emph{\ref{thm:LFTeValorAbsoluto}}.
This property enables the characterization of some open/closed balls
in terms of intervals, as detailed in \emph{Theorem} \emph{\ref{thm:BolaseIntervalos1}}.

\subsection{Nullifying sets}

~

When $\mathbb{T}$ is equipped with an order $\preceq$ compatible
with arithmetic operations, the set of positive elements $P_{\preceq}=\left\{ \alpha\in\mathbb{T}\,:\,0\prec\alpha\right\} $
is closed under sum and multiplication by positive scalars. The set
of non-zero $0$-symmetric numbers, $I_{0}$, plays a crucial role
in this framework.
\begin{prop}
\label{prop:CompatibilidadeComPositivos} If $\preceq$ is an order
on $\mathbb{T}$ compatible with arithmetic operations then $I_{0}\subseteq P_{\preceq}$
or $I_{0}\cap P_{\preceq}=\emptyset$. Furthermore:\\
i) $I_{0}\subseteq P_{\preceq}$ if and only if for all $\alpha\in\mathbb{T}$,
$\alpha=0$ or $0\prec\alpha$ or $0\prec-\alpha$. \\
ii) $I_{0}\cap P_{\preceq}=\emptyset$ if and only if for all $\alpha\in\mathbb{T}$,
$\alpha=0$ or $\alpha\prec0$ or $-\alpha\prec0$.
\end{prop}

\begin{proof}
Given $\alpha,\beta\in I_{0}$, there exists $r>0$ such that $r\alpha=\beta$.
By scalar multiplication compatibility, if $I_{0}\cap P_{\preceq}\neq\emptyset$
then $I_{0}\subseteq P_{\preceq}$. 

\emph{(i)} Suppose $I_{0}\subseteq P_{\preceq}$. Take $\alpha\neq0$
and $\alpha\notin P_{\preceq}$. Then $\alpha\prec0$ and, by sum
compatibility, $\alpha-\alpha\prec-\alpha$. Since $\alpha-\alpha\in I_{0}\subseteq P_{\preceq}$,
the transitivity of the order implies $0\prec-\alpha$. Conversely,
if the condition is satisfied then $I_{0}\subseteq P_{\preceq}$,
since $\alpha=-\alpha$ for all $\alpha\in I_{0}$.

\emph{(ii)} Suppose $I_{0}\cap P_{\preceq}=\emptyset$ and take $\alpha\in P_{\preceq}$.
Since $0\prec\alpha$, compatibility with the sum implies $-\alpha\prec\alpha-\alpha\prec0$
and by transitivity, $-\alpha\prec0$. The converse follows the same
idea given in \emph{(i)}.
\end{proof}
An order $\preceq$ on $\mathbb{T}$ is said to have \emph{positive
$0$-symmetric numbers} if there exists $\alpha\in I_{0}$ with $0\prec\alpha$.
The above result shows that assuming an arithmetic-compatible order
has positive $0$-symmetrics does not restrict generality ---it merely
selects between an order and its dual.

In decision-making contexts \cite{FRM}, orders with positive $0$-symmetric
numbers are associated with risk-prone behavior, whereas those where
$I_{0}\cap P_{\preceq}=\emptyset$ correspond to risk-averse decision-makers.
This is explored further in \emph{Example} \emph{\ref{exa:OrdensPessimistaeOtimista}}.
\begin{thm}
\label{th:OrdemNosAnuladores} Let $\preceq$ be a total order on
$\mathbb{T}$ compatible with arithmetic operations. Given $\alpha',\alpha''\in N_{ull}\left(\alpha\right)$
for any $\alpha\in\mathbb{T}$ then:\\
i) if $I_{0}\subseteq P_{\preceq}$, then $\alpha'\preceq\alpha''$
if and only if $sup\,S_{upp}\left(\alpha'\right)\leq sup\,S_{upp}\left(\alpha''\right)$;
or\\
ii) if $I_{0}\cap P_{\preceq}=\emptyset$, then $\alpha'\preceq\alpha''$
if and only if $sup\,S_{upp}\left(\alpha'\right)\geq sup\,S_{upp}\left(\alpha''\right)$.
That is, if and only if $inf\,S_{upp}\left(\alpha'\right)\leq inf\,S_{upp}\left(\alpha''\right)$.

Consequently, each nullifying set $\left(N_{ull}\left(\alpha\right),\preceq|_{N_{ull}\left(\alpha\right)}\right)$
has either a minimum (if $I_{0}\subseteq P_{\preceq}$ ) or a maximum
(if $I_{0}\cap P_{\preceq}=\emptyset$).
\end{thm}

\begin{proof}
Suppose $I_{0}\subseteq P_{\preceq}$. The following statements are
equivalent for $\alpha',\alpha''\in N_{ull}\left(\alpha\right)$:\\
a) $sup\,S_{upp}\left(\alpha'\right)<sup\,S_{upp}\left(\alpha''\right)$;
\\
b) there exists a scalar $t>0$ such that $sup\,S_{upp}\left(\alpha''\right)=sup\,S_{upp}\left(\alpha'\right)+t$
and $inf\,S_{upp}\left(\alpha''\right)=inf\,S_{upp}\left(\alpha'\right)-t$;\\
c) there is $\delta\in I_{0}$ such that $\alpha''=\alpha'+\delta$;
\\
d) $\alpha'\prec\alpha''$.

Implication (a) $\Longrightarrow$ (b) follows from the fact that
$inf\,S_{upp}\left(\alpha''\right)+sup\,S_{upp}\left(\alpha''\right)=inf\,S_{upp}\left(\alpha'\right)+sup\,S_{upp}\left(\alpha'\right)$
(see (\ref{eq:Anuladores}). Implication (b) $\Longrightarrow$ (c)
is straightforward. Implication (c) $\Longrightarrow$ (d) follows
from compatibility with the sum and the fact that $I_{0}\subseteq P_{\preceq}$.
And (d) implies (a) since if $\alpha'\preceq\alpha''$ then $sup\,S_{upp}\left(\alpha'\right)\leq sup\,S_{upp}\left(\alpha''\right)$,
otherwise the implications (a) $\Longrightarrow$ (b) $\Longrightarrow$
(c) lead us to a contradiction. This closes the case $I_{0}\subseteq P_{\preceq}$.

If $I_{0}\cap P_{\preceq}=\emptyset$, then the result follows from
duality since $\preceq^{*}$ satisfies \emph{(i)}. 
\end{proof}
In \emph{Section} \emph{\ref{sec:MIN-MAX-compatibilidadeEOrdemSomaSuperior}}
we continue to deal with the orders on the fibers of the natural projection.
For an arithmetic-compatible order $\preceq$, the presence of positive
$0$-symmetric numbers determines a key property: the ranking of elements
within each nullifying set $N_{ull}\left(\alpha\right)$ coincides
with the standard ordering of their support suprema ( $sup\,S_{upp}$$\left(\cdot\right)$).
If $\preceq$ has positive $0$-symmetric numbers, we denote the minimum
of $N_{ull}\left(\alpha\right)$ by $min\left(N_{ull}\left(\alpha\right)\right)$.
If $\preceq^{*}$ (the dual order) has positive $0$-symmetrics, we
denote the maximum by $max\left(N_{ull}\left(\alpha\right)\right)$.
\begin{rem}
\label{obs:ElementosMinimaisAnuladores} Let $\preceq$ be an arithmetic-compatible
order having positive $0$-symmetric numbers. Given $\alpha=\left(a_{1},a,a_{2}\right)\in\mathbb{T}$,
\[
min\left(N_{ull}\left(\alpha\right)\right)=\underset{\preceq}{min}\left\{ \left(x,a,y\right)\in\mathbb{T}\,:\,x+y=a_{1}+a_{2}\right\} .
\]
 Hence, if $a_{1}+a_{2}\leq2a$ then $min\left(N_{ull}\left(\alpha\right)\right)=\left(a_{1}+a_{2}-a,a,a\right)$.
If $a_{1}+a_{2}\geq2a,$ then $min\left(N_{ull}\left(\alpha\right)\right)=\left(a,a,a_{1}+a_{2}-a\right)$.
If $\preceq$ has no positive $0$-symmetric elements, these become
maxima instead.
\end{rem}

\begin{rem}
\label{obs:Min(-alpha)} In arithmetic-compatible orders: $min\left(N_{ull}\left(-\alpha\right)\right)=-min\left(N_{ull}\left(\alpha\right)\right)$
(or $max\left(N_{ull}\left(-\alpha\right)\right)=-max\left(N_{ull}\left(\alpha\right)\right)$,
when applicable).
\end{rem}

\begin{cor}
\label{cor:PropriedadeNormaLTP} Let $\preceq$ be an arithmetic-compatible
order. Given $\alpha,\beta\in\mathbb{T}$ with $\alpha\preceq\beta$
then:\\
i) if $I_{0}\subseteq P_{\preceq}$, then $\beta-\alpha\succeq0$;
and\\
ii) dually, if $I_{0}\cap P_{\preceq}=\emptyset$, then $\alpha-\beta\preceq0$. 

And $\beta-\alpha=0$ is valid only in the case $\beta=\alpha\in\mathbb{R}$. 
\end{cor}

\begin{proof}
Suppose $I_{0}\subseteq P_{\preceq}$. Since $\alpha\preceq\beta$,
sum compatibility implies $\alpha-\alpha\preceq\beta-\alpha$. Since
$\alpha-\alpha\in N_{ull}\left(0\right)=I_{0}\cup\left\{ 0\right\} $,
transitivity implies $\beta-\alpha\succeq0$. Dually, if $I_{0}\cap P_{\preceq}=\emptyset$,
then sum compatibility implies that $\alpha-\beta\preceq\beta-\beta$
and $\beta-\beta\in N_{ull}\left(0\right)$. By transitivity $\alpha-\beta\preceq0$.
\end{proof}
The converse of the above corollary is not true in general: if $I_{0}\subseteq P_{\preceq}$
and $\beta-\alpha\succ0$, then compatibility with sum implies that
$\beta+\alpha-\alpha\succ\alpha$ but does not guarantee $\beta\succ\alpha$.
We also have no guarantee that $\alpha-\beta\prec0$. For example,
if $\beta\in\alpha+I_{0}$ then $\beta-\alpha\succ0$ and $\alpha-\beta\succ0$.
And obviously, one of the conditions $\beta\succ\alpha$ or $\alpha\succ\beta$
is false.
\begin{example}
\label{exa:Ordensijk} Consider the lexicographic orders $\leq_{ijk}$
(where $\left\{ i,j,k\right\} =\left\{ 1,2,3\right\} $) given by
\begin{align*}
\left(a_{1},a_{2},a_{3}\right) & \leq_{ijk}\left(b_{1},b_{2},b_{3}\right)\,\Longleftrightarrow\,\left\{ \begin{array}{l}
\textrm{\ensuremath{a_{i}<b_{i}}; or}\\
\textrm{\ensuremath{a_{i}=b_{i}} and \ensuremath{a_{j}<b_{j}}; or}\\
\textrm{\ensuremath{a_{i}=b_{i}} and \ensuremath{a_{j}=b_{j}} and \ensuremath{a_{k}\leq b_{k}}.}
\end{array}\right.
\end{align*}

All orders $\leq_{ijk}$ are compatible with arithmetic operations.
Of these, $\le_{231}$, $\le_{312}$ and $\le_{321}$ have positive
$0$-symmetric numbers. Considering $\alpha=\left(-1,2,3\right)$
we have $\alpha,-\alpha\prec0$ for $\prec\in\left\{ <_{123},<_{132}\right\} $
and $-\alpha\prec\alpha$. While $0\prec\alpha,-\alpha$ for $\prec\in\left\{ <_{312},<_{321}\right\} $
and we also have $-\alpha\prec\alpha$. It is straightforward to find
similar examples for the orders $\leq_{213}$ and $\leq_{231}$. For
example, if $\beta=\left(-1,0,2\right)$ then $\beta,-\beta<_{213}0$
and $0<_{231}\beta,-\beta$.

\end{example}

\subsection{Weak law of trichotomy and the set of positives}

~

For arithmetic-compatible orders, the WLT prevents the \textquotedbl shuffling\textquotedbl{}
of nullifying sets, as demonstrated in \emph{Proposition} \emph{\ref{prop:PropriedadeReciproca}}.
The main consequence of \emph{Theorem \ref{thm:Compat.Op.AritmeticaseLFTePositivos}}
is that arithmetic-compatible orders satisfying the WLT are entirely
determined by their set of positive numbers.

The following result is crucial for the proof of \emph{Theorem \ref{thm:Compat.Op.AritmeticaseLFTePositivos}}.
Part \emph{(i)} establishes the converse of \emph{Corollary} \emph{\ref{cor:PropriedadeNormaLTP}}.
\begin{prop}
\label{prop:PropriedadeReciproca} Let $\preceq$ be an arithmetic-compatible
order satisfying the WLT. For $\alpha,\beta\in\mathbb{T}$ with $N_{ull}\left(\alpha\right)\neq N_{ull}\left(\beta\right)$:\\
i) $\alpha\prec\beta$~ $\Longleftrightarrow$ ~$0\prec\beta-\alpha$~
$\Longleftrightarrow$~ $\alpha-\beta\prec0$~ $\Longleftrightarrow$~
$-\beta\prec-\alpha$; and \\
ii) If $\alpha\prec\beta$, then $\alpha\prec\beta'$  for all $\beta'\in N_{ull}\left(\beta\right)$.
Just as, if $\alpha\succ\beta$, then $\alpha\succ\beta'$ for all
$\beta'\in N_{ull}\left(\beta\right)$.
\end{prop}

\begin{proof}
Suppose $I_{0}\subseteq P_{\preceq}$.

\emph{(i)} If $\alpha\prec\beta$, \emph{Corollary} \emph{\ref{cor:PropriedadeNormaLTP}.(i)}
implies that $\beta-\alpha\succ0$. Since $\beta-\alpha\notin N_{ull}\left(0\right)$,
$\alpha-\beta\prec0$ because of WLT. By sum compatibility, $\alpha-\alpha-\beta\prec-\alpha$.
Since $\alpha-\alpha\in N_{ull}\left(0\right)$, $\alpha-\alpha-\beta\succeq-\beta,$
and the result follows by transitivity. 

Conversely, from the fact that $-\beta\prec-\alpha$, \emph{Corollary}
\emph{\ref{cor:PropriedadeNormaLTP}.(i)} implies that $0\prec-\alpha-\left(-\beta\right)=\beta-\alpha$.
Since $\beta-\alpha\notin N_{ull}\left(0\right)$, $\alpha-\beta\prec0$
because of WLT. Then $\alpha\preceq\alpha+\beta-\beta\preceq\beta$.

\emph{(ii) }First, assume $\alpha\notin N_{ull}\left(0\right)$. If
$\alpha\prec0$, then by transitivity, $\alpha\prec\delta$ for all
$\delta\in N_{ull}\left(0\right)$. If $0\prec\alpha$ and $\delta\in N_{ull}\left(0\right)$,
then by WLT and transitivity, $-\alpha\prec0\preceq\delta$. By \emph{(i)},
$\delta=-\delta\prec\alpha$.

More generally, let $\alpha\notin N_{ull}\left(\beta\right)$ and
$\beta'\in N_{ull}\left(\beta\right)$. If $\alpha\prec\beta$, then
$\alpha-\beta'\prec\beta-\beta'$ by sum compatibility. Since $\beta-\beta'\in N_{ull}\left(0\right)$
and $\alpha-\beta'\notin N_{ull}\left(0\right)$, the previous argument
implies $\alpha-\beta'\prec0$. By \emph{(i)}, $\alpha\prec\beta'$.
If $\beta\prec\alpha$, then $\beta-\beta'\prec\alpha-\beta'$, and
by transitivity, $0\prec\alpha-\beta'$. By \emph{(i)}, $\beta'\prec\alpha$.

If $I_{0}\cap P_{\preceq}=\emptyset$, the result follows by duality,
provided $\preceq^{*}$ is an order compatible with arithmetic operations,
satisfies the WLT, and $I_{0}\subseteq P_{\preceq^{*}}$. 
\end{proof}
\begin{cor}
\label{cor:AnuladoresSaoIntervalos} In an arithmetic-compatible order
satisfying the WLT, each nullifying set is an interval, as is the
set $I_{0}$. 
\end{cor}

\begin{thm}
\label{thm:Compat.Op.AritmeticaseLFTePositivos} Let $\preceq_{1},\preceq_{2}$
be arithmetic-compatible orders on $\mathbb{T}$, with $\preceq_{2}$
satisfying the WLT. Then they have the same set of positives if and
only if they are equal (i.e. $\alpha\preceq_{1}\beta$ $\Longleftrightarrow$
$\alpha\preceq_{2}\beta$).
\end{thm}

\begin{proof}
Suppose $\preceq_{1},\preceq_{2}$ have the same set of positives,
$P_{\preceq_{1}}=P_{\preceq_{2}}$. In this case, either they both
have positive $0$-symmetric numbers or neither does. If $\preceq_{1}\neq\preceq_{2}$,
then there exist $\alpha,\beta\in\mathbb{T}$ such that $\alpha\preceq_{1}\beta$
but $\alpha\npreceq_{2}\beta$. By the totality of the order, we have
$\beta\prec_{2}\alpha\preceq_{1}\beta\prec_{2}\alpha$, and consequently,
by \emph{Theorem} \emph{\ref{th:OrdemNosAnuladores}}, $N_{ull}\left(\alpha\right)\neq N_{ull}\left(\beta\right)$.

If the orders have positive $0$-symmetric numbers, then sum compatibility
implies $\beta-\alpha\prec_{2}\alpha-\alpha\preceq_{1}\beta-\alpha$.
By transitivity, $\beta-\alpha\in P_{\preceq_{1}}$, but by \emph{Proposition}
\emph{\ref{prop:PropriedadeReciproca}.(ii)}, $\beta-\alpha\notin P_{\preceq_{2}}$.
Therefore $P_{\preceq_{1}}\neq P_{\preceq_{2}}$, contradicting the
initial hypothesis.

If neither order has positive $0$-symmetric numbers, then sum compatibility
implies $\alpha-\beta\preceq_{1}\beta-\beta\prec_{2}\alpha-\beta$.
By transitivity, $\alpha-\beta\notin P_{\preceq_{1}}$, since $\beta-\beta\prec0$.
By \emph{Proposition} \emph{\ref{prop:PropriedadeReciproca}}, $\alpha-\beta\in P_{\preceq_{2}}$.
Therefore, $P_{\preceq_{1}}\neq P_{\preceq_{2}}$, again contradicting
the initial hypothesis.
\end{proof}

\subsection{Fuzzy absolute value, fuzzy distance and intervals}

~

This section demonstrates that any arithmetic-compatible order $\preceq$
satisfying the Weak Law of Trichotomy (WLT) induces, up to duality,
a fuzzy absolute value. We conclude by writing certain open/closed
balls in terms of intervals.
\begin{thm}
\label{thm:LFTeValorAbsoluto} Let $\preceq$ be an order compatible
with arithmetic operations, satisfying the WLT and having positive
$0$-symmetric numbers. Then $\preceq$ induces a fuzzy absolute value,
$\left|\cdot\right|_{\preceq}$, that satisfying for arbitrary elements
of $\mathbb{T}$:\\
i) $\left|\alpha\right|_{\preceq}=\alpha$ if and only if $0\preceq\alpha$.
And $\left|\alpha\right|_{\preceq}=0$ if and only if $\alpha=0$;\\
ii) $\left|t\alpha\right|_{\preceq}=\left|t\right|\left|\alpha\right|_{\preceq}$
($t\in\mathbb{R}$), extending the real absolute value;\\
iii) $\left|\alpha+\beta\right|_{\preceq}\preceq\left|\alpha\right|_{\preceq}+\left|\beta\right|_{\preceq}$;\\
iv) $\left|\alpha-\gamma\right|_{\preceq}\preceq\left|\alpha-\beta\right|_{\preceq}+\left|\beta-\gamma\right|_{\preceq}$;\\
v) $\left|\left|\alpha\right|_{\preceq}-\left|\beta\right|_{\preceq}\right|_{\preceq}\preceq\left|\alpha-\beta\right|_{\preceq}$.
\end{thm}

\begin{proof}
\emph{(i) }If $\alpha\in N_{ull}\left(0\right)$, then $\alpha=-\alpha$,
and hence $\left|\alpha\right|_{\preceq}=\alpha\succeq0$. If $\alpha\notin N_{ull}\left(0\right)$
then either $0\prec\alpha$ or $0\prec-\alpha$, and only one of these
alternatives by WLT. If $0\prec\alpha$ then $\left|\alpha\right|_{\preceq}=\alpha$,
since $-\alpha\prec0$ by WLT. Similarly, if $0\prec-\alpha$, then
$\left|\alpha\right|_{\preceq}=-\alpha$.

\emph{(ii) }If $\alpha\in N_{ull}\left(0\right)$, then $t\alpha\in N_{ull}\left(0\right)$
for all $t\in\mathbb{R}$. In this case, $\left|t\alpha\right|_{\preceq}=t\alpha=-t\alpha=\left|t\right|\left|\alpha\right|_{\preceq}$.
Suppose $\alpha\notin N_{ull}\left(0\right)$. If $0\prec\alpha$
and $0<t$, then $0\prec t\alpha$, by scalar multiplication compatibility,
and $\left|t\alpha\right|_{\preceq}=t\alpha=\left|t\right|\left|\alpha\right|_{\preceq}$.
If $0\prec\alpha$ and $0>t$, then $0\prec-t\alpha$ by scalar multiplication
compatibility, and consequently $t\alpha\prec0$ by WLT. Therefore,
$\left|t\alpha\right|_{\preceq}=\left(-t\right)\alpha=\left|t\right|\left|\alpha\right|_{\preceq}$.
The case $\alpha\prec0$ is analogous.

\emph{(iii) }$\left|\alpha+\beta\right|_{\preceq}=\pm\left(\alpha+\beta\right)\preceq\left|\alpha\right|_{\preceq}+\left|\beta\right|_{\preceq}$,
by sum compatibility, provided $\pm\alpha\preceq\left|\alpha\right|_{\preceq}$
and $\pm\beta\preceq\left|\beta\right|_{\preceq}$.

\emph{(iv) }For the triangle inequality, consider any $\alpha,\beta,\gamma\in\mathbb{T}$.
We analyze all cases directly:

\emph{1)} If $\alpha\preceq\beta\preceq\gamma$: by \emph{Corollary}
\emph{\ref{cor:PropriedadeNormaLTP}}, $0\preceq\beta-\alpha$ and
$0\preceq\gamma-\beta$. By \emph{(i)},
\[
\left|\alpha-\beta\right|_{\preceq}+\left|\beta-\gamma\right|_{\preceq}=\beta-\alpha+\gamma-\beta=\gamma-\alpha+\beta-\beta\succeq\gamma-\alpha=\left|\alpha-\gamma\right|_{\preceq},
\]
 since $\beta-\beta\succeq0$ and the order is compatible with sum.

\emph{2)} If $\alpha\preceq\gamma\preceq\beta$: here $0\preceq\beta-\alpha$
and $0\preceq\beta-\gamma$. By \emph{(i)}, 
\[
\left|\alpha-\beta\right|_{\preceq}+\left|\beta-\gamma\right|_{\preceq}=\beta-\alpha+\beta-\gamma\succeq\beta-\alpha.
\]
 Since $\gamma\preceq\beta$, we have $\beta-\alpha\succeq\gamma-\alpha=\left|\alpha-\gamma\right|_{\preceq}$.

\emph{3)} If $\beta\preceq\alpha\preceq\gamma$: here $0\preceq\alpha-\beta$
and $0\preceq\gamma-\beta$. By \emph{(i)},
\[
\left|\alpha-\beta\right|_{\preceq}+\left|\beta-\gamma\right|_{\preceq}=\alpha-\beta+\gamma-\beta\succeq\gamma-\beta
\]
If $\alpha-\beta=0$ then $\gamma-\beta=\gamma-\alpha=\left|\alpha-\gamma\right|_{\preceq}$.
If $\alpha-\beta\succ0$, $\beta-\alpha\prec0$ or $\alpha\in N_{ull}\left(\beta\right)$
because of WLT. If $\beta-\alpha\prec0$, \emph{Proposition} \emph{\ref{prop:PropriedadeReciproca}.(i)}
implies that $-\alpha\preceq-\beta$, hence $\gamma-\beta\succeq\gamma-\alpha=\left|\alpha-\gamma\right|_{\preceq}$.
For contradiction, assume $\alpha\in N_{ull}\left(\beta\right)$ and
$\alpha-\beta+\gamma-\beta\prec\gamma-\alpha$. Then by sum compatibility
$\alpha-\beta+\gamma+\beta-\beta\prec\gamma+\beta-\alpha$. Since
$\alpha-\beta=\beta-\alpha$, we get $\gamma+\beta-\beta\prec\gamma$
- a contradiction as $\beta-\beta\succ0$. Then $\alpha-\beta+\gamma-\beta\succ\gamma-\alpha=\left|\alpha-\gamma\right|_{\preceq}$.

\emph{4) }If $\beta\preceq\gamma\preceq\alpha$: then $0\preceq\alpha-\beta$
and $0\preceq\gamma-\beta$. By \emph{(i)},
\[
\left|\alpha-\beta\right|_{\preceq}+\left|\beta-\gamma\right|_{\preceq}=\alpha-\beta+\gamma-\beta\succeq\alpha-\beta.
\]
 If $\beta\notin N_{ull}\left(\gamma\right)$, \emph{Proposition}
\emph{\ref{prop:PropriedadeReciproca}.(i)} gives $-\beta\succ-\gamma$,
so $\alpha-\beta\succeq\alpha-\gamma=\left|\alpha-\gamma\right|_{\preceq}$,
by sum compatibility. If $\beta\in N_{ull}\left(\gamma\right)$ and
$\alpha-\beta+\gamma-\beta\prec\alpha-\gamma$, then $\alpha-\beta+\gamma-\beta+\beta\prec\alpha-\gamma+\beta$.
Since $\gamma-\beta=\beta-\gamma$, we have $\alpha-\beta+\gamma-\beta+\beta\prec\alpha-\gamma+\beta$
and hence $\alpha+\beta-\beta\prec\alpha$ - a contradiction. Thus
$\alpha-\beta+\gamma-\beta\succeq\alpha-\gamma=\left|\alpha-\gamma\right|_{\preceq}$.

\emph{5) }If $\gamma\preceq\alpha\preceq\beta$: in this case $0\preceq\beta-\alpha$
and $0\preceq\beta-\gamma$. By \emph{(i)},
\[
\left|\alpha-\beta\right|_{\preceq}+\left|\beta-\gamma\right|_{\preceq}=\beta-\alpha+\beta-\gamma\succeq\beta-\gamma\succeq\alpha-\gamma=\left|\alpha-\gamma\right|_{\preceq}.
\]

\emph{6) }If $\gamma\preceq\beta\preceq\alpha$: in this case $0\preceq\alpha-\beta$
and $0\preceq\beta-\gamma$. By \emph{(i)},
\[
\left|\alpha-\beta\right|_{\preceq}+\left|\beta-\gamma\right|_{\preceq}=\alpha-\beta+\beta-\gamma\succeq\alpha-\gamma=\left|\alpha-\gamma\right|_{\preceq}.
\]

\emph{(v) }Without loss of generality, assume $\alpha\preceq\beta$.
Then $\left|\alpha-\beta\right|_{\preceq}=\beta-\alpha$, by \emph{(i)}.
If $\left|\alpha\right|_{\preceq}=\alpha$, then $\left|\beta\right|_{\preceq}=\beta$,
so $\left|\left|\alpha\right|_{\preceq}-\left|\beta\right|_{\preceq}\right|_{\preceq}=\left|\alpha-\beta\right|_{\preceq}$.
If $\left|\alpha\right|_{\preceq}=-\alpha$, then $\left|\left|\alpha\right|_{\preceq}-\left|\beta\right|_{\preceq}\right|_{\preceq}=\left|\left|\beta\right|_{\preceq}+\alpha\right|_{\preceq}$.
When $\left|\beta\right|_{\preceq}=-\beta$, we have $\left|\left|\beta\right|_{\preceq}+\alpha\right|_{\preceq}=\left|\alpha-\beta\right|_{\preceq}$.
If $\left|\beta\right|_{\preceq}=\beta$ then $\left|\left|\beta\right|_{\preceq}+\alpha\right|_{\preceq}=\left|\alpha+\beta\right|_{\preceq}\in\left\{ \beta+\alpha,-\alpha-\beta\right\} $.
In both cases, $\left|\left|\alpha\right|_{\preceq}-\left|\beta\right|_{\preceq}\right|_{\preceq}\preceq\beta-\alpha=\left|\alpha-\beta\right|_{\preceq}$.
\end{proof}
\begin{example}
The lexicographic orders $\leq_{ijk}$ (where $\left\{ i,j,k\right\} =\left\{ 1,2,3\right\} $)
fail to satisfy the weak law of trichotomy (WLT), as demonstrated
in \emph{Example} \emph{\ref{exa:Ordensijk}}. However, the specific
order $\leq_{231}$ induces a fuzzy absolute value $\left|\cdot\right|_{\le_{231}}$,
as will be shown in \emph{Example} \emph{\ref{exa:Ordem231InduzValorAbsolutoDifuso}}.
\end{example}

From now until the end of this section, $\preceq$ denotes an arithmetic-compatible
order having positive $0$-symmetric numbers and satisfying the WLT.
As a consequence of \emph{Theorem} \emph{\ref{thm:LFTeValorAbsoluto}},
$\preceq$ induces a fuzzy distance, $D_{\preceq}$, satisfying for
any triangular fuzzy numbers:\\
\emph{i)} $0\preceq D_{\preceq}\left(\alpha,\beta\right)$ and $0=D_{\preceq}\left(\alpha,\beta\right)$
if and only if $\alpha=\beta\in\mathbb{R}$; \\
\emph{ii)} $D_{\preceq}\left(\alpha,\beta\right)=D_{\preceq}\left(\beta,\alpha\right)$;\\
\emph{iii)} $D_{\preceq}\left(\alpha,\gamma\right)\preceq D_{\preceq}\left(\alpha,\beta\right)+D_{\preceq}\left(\beta,\gamma\right)$.

In \emph{(i)} we see a diffuse characteristic of distance: 
\begin{itemize}
\item $0=D_{\preceq}\left(\alpha,\beta\right)$ if and only if $\alpha=\beta\in\mathbb{R}$.
Furthermore, $D_{\preceq}\left(\alpha,\alpha\right)\in N_{ull}\left(0\right)$
and $D_{\preceq}\left(\alpha,\alpha\right)=0$ if and only if $\alpha\in\mathbb{R}$.
That is, the distance between a triangular fuzzy number and itself
is practically zero. It is zero only in the case where the number
is a real number.
\end{itemize}
We also highlight:
\begin{itemize}
\item Given $\alpha\in\mathbb{T}$, denote by $\alpha_{0}=min\left(N_{ull}\left(\alpha\right)\right)$,
guaranteed by \emph{Theorem} \emph{\ref{th:OrdemNosAnuladores}}.
Then $D_{\preceq}\left(\alpha,\alpha\right)\succeq D_{\preceq}\left(\alpha_{0},\alpha\right)$,
where equality holds only in the case where $\alpha=\alpha_{0}$.
This is because if $\alpha=\alpha_{0}+\delta$, with $\delta\in I_{0}$,
then $D_{\preceq}\left(\alpha,\alpha\right)=D_{\preceq}\left(\alpha_{0},\alpha\right)+\delta$.
In other words, $\alpha_{0}$ is the point with the shortest distance
to each point in $N_{ull}\left(\alpha\right)$.
\end{itemize}
From now on, we prefer to write $\left|\alpha-\beta\right|_{\preceq}$
instead of $D_{\preceq}\left(\alpha,\beta\right)$. We conclude this
section by characterizing certain open/closed balls through intervals.
For this, we need the following result: 
\begin{lem}
\label{lem:SolucoesEquacoesDiferencas} Given $\gamma=\left(c_{1},c,c_{2}\right)$
and $\beta=\left(b_{1},b,b_{2}\right)$ triangular fuzzy numbers,\\
i) the equation $\beta-\alpha=\gamma$ has a solution for some $\alpha\in\mathbb{T}$
if and only if $b-b_{1}\leq c-c_{1}$ (the lower margin of $\beta$
is bounded by that of $\gamma$) and $b_{2}-b\leq c_{2}-c$ (the upper
margin of $\beta$ is bounded by that of $\gamma$). In this case,
$\alpha=\left(b_{2}-c_{2},b-c,b_{1}-c_{1}\right)$ is the solution.\\
ii) the equation $\alpha-\beta=\gamma$ has a solution for some $\alpha\in\mathbb{T}$
if and only if $b-b_{1}\leq c_{2}-c$ and $b_{2}-b\leq c-c_{1}$(the
lower margin of $\beta$ is bounded by the upper margin of $\gamma$
and the upper margin of $\beta$ by the lower margin of $\gamma$).
In this case, $\alpha=\left(b_{2}+c_{1},b+c,b_{1}+c_{2}\right)$ is
the solution.
\end{lem}

\begin{proof}
Suppose $\alpha=\left(a_{1},a,a_{2}\right)$ solution of the equation
$\beta-\alpha=\gamma$. Then $a_{1}=b_{2}-c_{2}$, $a=b-c$ and $a_{2}=b_{1}-c_{1}$
with $a_{1}\leq a\leq a_{2}$. Therefore $b-b_{1}\leq c-c_{1}$ and
$b_{2}-b\leq c_{2}-c$. The converse is also direct since if $b-b_{1}\leq c-c_{1}$
and $b_{2}-b\leq c_{2}-c$, then $\alpha=\left(b_{2}-c_{2},b-c,b_{1}-c_{1}\right)$
is a triangular fuzzy number and satisfies the equation. Item \emph{(ii)}
is analogous.
\end{proof}
\begin{itemize}
\item Unlike the usual absolute value in real numbers, given $\beta,\gamma\in\mathbb{T}$
with $\gamma\succ0$ the equation $\left|\alpha-\beta\right|_{\preceq}=\gamma$
does not always admit a solution and, even when it does, it can satisfy
only one of the equations in the previous lemma. 
\end{itemize}
In the next result we write the balls centered at $\beta$ of radius
$\gamma$ in terms of intervals for the cases where $\beta,\gamma$
satisfy both conditions \emph{(i)} and \emph{(ii)} of \emph{Lemma}
\emph{\ref{lem:SolucoesEquacoesDiferencas}}. The remaining cases
are established in \emph{Theorem \ref{thm:BolasEIntervalos2}}.
\begin{thm}
\label{thm:BolaseIntervalos1} Given $\gamma=\left(c_{1},c,c_{2}\right)\succ0$
and $\beta=\left(b_{1},b,b_{2}\right)\in\mathbb{T}$ with $max\left\{ b-b_{1},b_{2}-b\right\} \leq min\left\{ c-c_{1},c_{2}-c\right\} $:\\
i) if $\gamma=\left(-k,0,k\right)\in I_{0}$, $k>0$, $\left|\alpha-\beta\right|_{\preceq}=\gamma$
has a unique solution $\alpha_{0}=\left(b_{2}-k,b,b_{1}+k\right)$
and $\overline{\mathcal{B}_{\preceq}}\left(\beta,\gamma\right)=\left[\widehat{\beta},\alpha_{0}\right]_{\preceq}$,
where $\widehat{\beta}=min\left(N_{ull}\left(\beta\right)\right)$.\\
ii) if $\gamma\notin I_{0}$, $\left|\alpha-\beta\right|_{\preceq}=\gamma$
has two solutions $\alpha_{1}=\left(b_{2}-c_{2},b-c,b_{1}-c_{1}\right)$
and $\alpha_{2}=\left(b_{2}+c_{1},b+c,b_{1}+c_{2}\right)$, with $\alpha_{1}\prec\beta\prec\alpha_{2}$.
In this case, 
\begin{align*}
\overline{\mathcal{B}_{\preceq}}\left(\beta,\gamma\right) & =\left(\left[\alpha_{1},\alpha_{2}\right]_{\preceq}\cup N_{ull}\left(\alpha_{1}\right)\right)\backslash\left(\alpha_{1}+I_{0}\right)\\
 & =\left[min\left(N_{ull}\left(\alpha_{1}\right)\right),\alpha_{2}\right]_{\preceq}\backslash\left(\alpha_{1}+I_{0}\right).
\end{align*}
\end{thm}

\begin{proof}
\emph{Case} \emph{(i)}: by hypothesis, $max\left\{ b-b_{1},b_{2}-b\right\} \leq k$.
Both conditions of \emph{Lemma} \emph{\ref{lem:SolucoesEquacoesDiferencas}}
are satisfied. Since $\gamma\in I_{0}$, we have $\alpha-\beta=\beta-\alpha=\gamma=-\gamma$,
and therefore the same $\alpha$ satisfies both conditions. Consequently,
$\left|\alpha-\beta\right|_{\preceq}=\gamma$ has a unique solution
$\alpha_{0}=\left(b_{2}-k,b,b_{1}+k\right)\in N_{ull}\left(\beta\right)$.

Let $\widehat{\beta}=min\left(N_{ull}\left(\beta\right)\right)$.
If $\widehat{\beta}\preceq\alpha\preceq\alpha_{0}$, then $\alpha-\beta\in N_{ull}\left(0\right)$,
and hence $\left|\alpha-\beta\right|_{\preceq}=\alpha-\beta=\beta-\alpha$.
By sum compatibility, $\widehat{\beta}-\beta\preceq\alpha-\beta\preceq\alpha_{0}-\beta=\gamma$,
and thus $\alpha\in\overline{\mathcal{B}_{\preceq}}\left(\beta,\gamma\right)$.
For $\alpha\prec\widehat{\beta}$, we have $\alpha\prec\beta$ and
$\left|\alpha-\beta\right|_{\preceq}=\beta-\alpha\notin N_{ull}\left(0\right)$.
By \emph{Proposition \ref{prop:PropriedadeReciproca}.(ii)}, $\left|\alpha-\beta\right|_{\preceq}\succ\gamma$,
therefore $\alpha\notin\overline{\mathcal{B}_{\preceq}}\left(\beta,\gamma\right)$.
If $\alpha_{0}\prec\alpha$, sum compatibility implies $\gamma=\alpha_{0}-\beta\prec\alpha-\beta=\left|\alpha-\beta\right|_{\preceq}$.
Consequently, $\overline{\mathcal{B}_{\preceq}}\left(\beta,\gamma\right)=\left[\widehat{\beta},\alpha_{0}\right]_{\preceq}$.

\emph{Case (ii)}: if $\gamma\notin I_{0}$, \emph{Lemma} \emph{\ref{lem:SolucoesEquacoesDiferencas}}
guarantees that the equation $\left|\alpha-\beta\right|_{\preceq}=\gamma$
has two solutions $\alpha_{1}=\left(b_{2}-c_{2},b-c,b_{1}-c_{1}\right)$
and $\alpha_{2}=\left(b_{2}+c_{1},b+c,b_{1}+c_{2}\right)$, with $\alpha_{1}\prec\beta\prec\alpha_{2}$. 

If $\beta\preceq\alpha\prec\alpha_{2}$, sum compatibility implies
$\left|\alpha-\beta\right|_{\preceq}=\alpha-\beta\prec\alpha_{2}-\beta=\gamma$.
Hence $\alpha\in\overline{\mathcal{B}_{\preceq}}\left(\beta,\gamma\right)$.

If $\alpha_{1}\prec\alpha\prec\beta$ and $\alpha\notin N_{ull}\left(\alpha_{1}\right)$,
then $\left|\alpha-\beta\right|_{\preceq}=\beta-\alpha$, and by \emph{Proposition}
\emph{\ref{prop:PropriedadeReciproca}.(i),} $-\alpha\prec-min\left(N_{ull}\left(\alpha_{1}\right)\right)$.
By \emph{Observation} \emph{\ref{obs:Min(-alpha)}}, $-min\left(N_{ull}\left(\alpha_{1}\right)\right)=min\left(N_{ull}\left(-\alpha_{1}\right)\right)$,
and thus $-min\left(N_{ull}\left(\alpha_{1}\right)\right)\preceq-\alpha_{1}$.
Sum compatibility yields $\beta-\alpha\prec\beta-min\left(N_{ull}\left(\alpha_{1}\right)\right)\preceq\beta-\alpha_{1}=\gamma$.
Therefore, $\alpha\in\overline{\mathcal{B}_{\preceq}}\left(\beta,\gamma\right)$.

|Now suppose $\alpha\in N_{ull}\left(\alpha_{1}\right)$. If $min\left(N_{ull}\left(\alpha_{1}\right)\right)\preceq\alpha\preceq\alpha_{1}$,
then $-min\left(N_{ull}\left(\alpha_{1}\right)\right)\preceq-\alpha\preceq-\alpha_{1}$
by \emph{Theorem} \emph{\ref{th:OrdemNosAnuladores}} and \emph{Observation}
\ref{obs:ElementosMinimaisAnuladores}. Sum compatibility gives $\left|\alpha-\beta\right|_{\preceq}=\beta-\alpha\preceq\beta-\alpha_{1}=\gamma$,
hence $\alpha\in\overline{\mathcal{B}_{\preceq}}\left(\beta,\gamma\right)$.
If $\alpha_{1}\prec\alpha$, then $-\alpha_{1}\prec-\alpha$, and
sum compatibility implies $\gamma=\beta-\alpha_{1}\prec\beta-\alpha=\left|\alpha-\beta\right|_{\preceq}$.
In this case, $\alpha\notin\overline{\mathcal{B}_{\preceq}}\left(\beta,\gamma\right)$.

There remaining case is $\alpha\prec\alpha_{1}$ and $\alpha\notin N_{ull}\left(\alpha_{1}\right)$.
\emph{Proposition} \emph{\ref{prop:PropriedadeReciproca}.(i)} gives
$-\alpha_{1}\prec-\alpha$. Sum compatibility yields $\gamma=\beta-\alpha_{1}\prec\beta-\alpha=\left|\alpha-\beta\right|_{\preceq}$,
and consequently $\alpha\notin\overline{\mathcal{B}_{\preceq}}\left(\beta,\gamma\right)$.

Thus, in \emph{case (ii)} we have:
\begin{align*}
\overline{\mathcal{B}_{\preceq}}\left(\beta,\gamma\right) & =\left(\left[\alpha_{1},\alpha_{2}\right]_{\preceq}\cup N_{ull}\left(\alpha_{1}\right)\right)\backslash\left(\alpha_{1}+I_{0}\right)\\
 & =\left[min\left(N_{ull}\left(\alpha_{1}\right)\right),\alpha_{2}\right]_{\preceq}\backslash\left(\alpha_{1}+I_{0}\right).
\end{align*}
\end{proof}
Regarding open balls, in \emph{(i)} we have $\mathcal{B}_{\preceq}\left(\beta,\gamma\right)=\overline{\mathcal{B}_{\preceq}}\left(\beta,\gamma\right)\backslash\left\{ \alpha_{0}\right\} $
and in \emph{(ii)} we have $\mathcal{B}_{\preceq}\left(\beta,\gamma\right)=\overline{\mathcal{B}_{\preceq}}\left(\beta,\gamma\right)\backslash\left\{ \alpha_{1},\alpha_{2}\right\} $.
The previous result shows more fuzzy properties of the distance:
\begin{itemize}
\item In case \emph{(i)} above, if $max\left\{ b-b_{1},b_{2}-b\right\} >k$
then $\left|\alpha-\beta\right|_{\preceq}=\gamma$ has no solution
and $\overline{\mathcal{B}_{\preceq}}\left(\beta,\gamma\right)=\emptyset$.
\item Also in \emph{(i)}, if $b_{1}+k<b_{2}$ then $\beta\notin\overline{\mathcal{B}_{\preceq}}\left(\beta,\gamma\right)$. 
\end{itemize}

\section{\label{sec:MIN-MAX-compatibilidadeEOrdemSomaSuperior} Compatibility
with \emph{MIN-MAX} operators and the total sum order}

Here, we examine the implications of adding compatibility with \emph{MIN-MAX}
operators to the desired properties. Recall from \emph{Definition
\ref{def:OrdemRegular}} that an order compatible with both arithmetic
operations and\emph{ MIN-MAX} operators is called regular.

\emph{Theorem} \emph{\ref{thm:OrdensCompativeisSobreAsFibras}} characterizes
the regular orders that satisfy the WLT when restricted to the fibers
of the natural projection. Moreover, compatibility with \emph{MIN-MAX}
operators allows us to express---in terms of intervals---the open/closed
balls not covered by \emph{Theorem} \emph{\ref{thm:BolaseIntervalos1}},
as demonstrated in \emph{Theorem} \emph{\ref{thm:BolasEIntervalos2}}.

In \emph{Example} \ref{exa:Pre-ordemT}, we define the \emph{total
sum order} (\emph{\ref{eq:OrdemSomaTotal}}), which is regular, satisfies
the WLT, and has positive $0$-symmetric numbers.

Assume that $\preceq$ is an order relation compatible with \emph{MIN-MAX}
operators. As a consequence of \emph{Proposition \ref{prop:TFNcomparaveis}},
given a triangular fuzzy number $\alpha=\left(a_{1},a,a_{2}\right)$,
there exists:
\begin{itemize}
\item An unbounded vertical rectangular strip in $\pi^{-1}\left(a\right)$:
\[
\left\{ \left(x,a,y\right)\in\pi^{-1}\left(a\right)\,:\,\textrm{ \ensuremath{a_{1}\leq x\leq a} and \ensuremath{a_{2}\leq y} }\right\} 
\]
containing elements in $\pi^{-1}\left(a\right)$ that are greater
than or equal to $\alpha$ under the order $\preceq$; and
\item An unbounded horizontal rectangular strip in $\pi^{-1}\left(a\right)$:
\[
\left\{ \left(x,a,y\right)\in\pi^{-1}\left(a\right)\,:\,\textrm{ \ensuremath{x\leq a_{1}} and \ensuremath{a\leq y\leq a_{2}} }\right\} 
\]
containing elements in $\pi^{-1}\left(a\right)$ that are less than
or equal to $\alpha$.
\end{itemize}
If $\preceq$, in addition to being compatible with \emph{MIN-MAX}
operators, is also compatible with arithmetic operations, satisfies
WLT, and has positive $0$-symmetric elements, then:
\[
\left\{ \left(x,0,y\right)\in\pi^{-1}\left(0\right)\,:\,\textrm{ \ensuremath{x+y\geq0} e \ensuremath{0<y} }\right\} =\underset{0\leq t}{\bigcup}\left\{ \left(-t,0,t+k\right)\in\pi^{-1}\left(0\right)\,:\,\textrm{ \ensuremath{0\leq k} e \ensuremath{0<k+t} }\right\} \subset P_{\preceq}.
\]
 Since $\preceq$ extends the standard order on real numbers due to
its \emph{MIN-MAX} compatibility, for each $a>0$:
\[
a+\left[\left\{ \left(x,0,y\right)\in\pi^{-1}\left(0\right)\,:\,\textrm{ \ensuremath{x+y\geq0} e \ensuremath{0<y} }\right\} \cup\left\{ 0\right\} \right]=\left\{ \left(a_{1},a,a_{2}\right)\in\pi^{-1}\left(a\right)\,:\,\textrm{ \ensuremath{a_{1}+a_{2}\geq2a} }\right\} \subset P_{\preceq}.
\]

\begin{thm}
\label{thm:OrdensCompativeisSobreAsFibras} Let $\preceq$ be a regular
order satisfying WLT. Then for each $t\in\mathbb{R}$, the order $\preceq$
on each fiber $\pi^{-1}\left(t\right)$ is given by:\\
i) $\left(x_{1},t,y_{1}\right)\preceq\left(x_{2},t,y_{2}\right)\Longleftrightarrow\left\{ \begin{array}{l}
\textrm{\ensuremath{x_{1}+y_{1}<x_{2}+y_{2},} or }\\
\textrm{\ensuremath{x_{1}+y_{1}=x_{2}+y_{2}} and \ensuremath{y_{1}\leq y_{2}},}
\end{array}\right.$ if $I_{0}\subseteq P_{\preceq}$. Or \\
ii) $\left(x_{1},t,y_{1}\right)\preceq\left(x_{2},t,y_{2}\right)\Longleftrightarrow\left\{ \begin{array}{l}
\textrm{\ensuremath{x_{1}+y_{1}<x_{2}+y_{2},} or }\\
\textrm{\ensuremath{x_{1}+y_{1}=x_{2}+y_{2}} and \ensuremath{x_{1}\leq x_{2}},}
\end{array}\right.$ if $I_{0}\cap P_{\preceq}=\emptyset$. \\
\end{thm}

\begin{proof}
Take arbitrary $\alpha=\left(x_{1},t,y_{1}\right)$ and $\beta=\left(x_{2},t,y_{2}\right)$
in $\pi^{-1}\left(t\right)$ with $\left(x_{1},t,y_{1}\right)\preceq\left(x_{2},t,y_{2}\right)$.
Then $N_{ull}\left(\alpha\right)=N_{ull}\left(\beta\right)$ if and
only if $x_{1}+y_{1}=x_{2}+y_{2}$, and in this case \emph{Theorem
\ref{th:OrdemNosAnuladores}} states that $\alpha\preceq\beta$ if
and only if $y_{1}\leq y_{2}$.

Assume $x_{1}+y_{1}\neq x_{2}+y_{2}$ and without loss of generality
that $x_{1}+y_{1}<x_{2}+y_{2}$. There exists $s>0$ such that $x_{1}+y_{1}+s=x_{2}+y_{2}$.
Let $\gamma=\left(x_{1},t,y_{1}+s\right)$. By \emph{MIN-MAX} compatibility,
$\alpha\prec\gamma$. Since $\gamma\in N_{ull}\left(\beta\right)$,
\emph{Proposition} \emph{\ref{prop:PropriedadeReciproca}.(ii)} implies
$\alpha\prec\beta$.
\end{proof}
\begin{example}
\label{exa:OrdensPessimistaeOtimista} In \cite{FRM}, the authors
propose ranking methods for triangular fuzzy numbers considering decision-makers'
risk propensity. The pessimistic method (\emph{Section II} of \cite{FRM})
is designed for strongly risk-averse decision-makers, considering
only \textquotedbl worst-case outcomes\textquotedbl . \emph{Section
III} of \cite{FRM} proposes the optimistic method for strongly risk-seeking
decision-makers.

\emph{Equations} \emph{(2.2)-(2.4)} in \cite{FRM} show the pessimistic
method is equivalent to the total preorder:
\[
\left(a_{1},a,a_{2}\right)\preceq_{P}\left(b_{1},b,b_{2}\right)\Longleftrightarrow a_{1}+a\leq b_{1}+b.
\]
 It is straightforward to verify that this preorder is compatible
with arithmetic operations, with MIN-MAX operators, and satisfies
$\alpha\prec_{P}0$ for each $\alpha\in I_{0}$. Given $\alpha\in\mathbb{T}$,
\[
N_{ull}\left(\alpha\right)\nsubseteq\left\{ \beta\in\mathbb{T}\,:\,\beta\sim_{P}\alpha\right\} =\left\{ \beta\in\mathbb{T}\,:\,\textrm{\ensuremath{\beta\preceq_{P}\alpha} e \ensuremath{\beta\succeq_{P}\alpha}}\right\} .
\]
That is, this ranking method \textquotedbl shuffles\textquotedbl{}
the nullifying sets. Consequently, any extension of $\preceq_{P}$
to a regular total order on $\mathbb{T}$ cannot satisfy the weak
law of trichotomy, as evidenced by \emph{Proposition} \emph{\ref{prop:PropriedadeReciproca}.(ii)}.
Furthermore, no extension of $\preceq_{P}$ to a total order on $\mathbb{T}$
will induce a notion of fuzzy absolute value, provided $I_{0}\cap P_{\preceq_{P}}=\emptyset$. 

Among the numerous extensions of $\preceq_{P}$, we define the\emph{
pessimistic order} as follows:
\begin{equation}
\left(a_{1},a,a_{2}\right)\leq_{P}\left(b_{1},b,b_{2}\right)\,\Longleftrightarrow\,\left\{ \begin{array}{l}
\textrm{\ensuremath{a_{1}+a<b_{1}+b,} or }\\
\textrm{\ensuremath{a_{1}+a=b_{1}+b} and \ensuremath{a_{2}<b_{2}}; or }\\
\textrm{\ensuremath{a_{1}+a=b_{1}+b} and \ensuremath{a_{2}=b_{2}} and \ensuremath{a\leq b}.}
\end{array}\right.\label{eq:OrdemPessimista}
\end{equation}
Then $\leq_{P}$ is regular, with $I_{0}\cap P_{\leq_{P}}=\emptyset$.

Analogously, the optimistic method is equivalent to the following
total preorder: 
\[
\left(a_{1},a,a_{2}\right)\preceq_{O}\left(b_{1},b,b_{2}\right)\Longleftrightarrow a+a_{2}\leq b+b_{2}.
\]
We have that $\preceq_{O}$ is regular and possesses positive $0$-symmetric
elements. However, as in the pessimistic case, no extension to a regular
total order on $\mathbb{T}$ will satisfy the Weak Law of Trichotomy.
We define the \emph{optimistic order} as: 
\begin{equation}
\left(a_{1},a,a_{2}\right)\leq_{O}\left(b_{1},b,b_{2}\right)\,\Longleftrightarrow\,\left\{ \begin{array}{l}
\textrm{\ensuremath{a+a_{2}<b+b_{2},} or }\\
\textrm{\ensuremath{a+a_{2}=b+b_{2}} and \ensuremath{a_{1}<b_{1}}; or }\\
\textrm{\ensuremath{a+a_{2}=b+b_{2}} and \ensuremath{a_{1}=b_{1}} and \ensuremath{a\leq b}.}
\end{array}\right.\label{eq:OrdemOtmista}
\end{equation}
This is a regular order with positive $0$-symmetric elements that
does not satisfy the Weak Law of Trichotomy. The definitions of both
orders (\ref{eq:OrdemPessimista}) and (\ref{eq:OrdemOtmista}) are
based on the characterization of fibers $\pi^{-1}\left(t\right)$
(for $t\in\mathbb{R}$) of regular orders satisfying WLT, as given
in \emph{Theorem} \ref{thm:OrdensCompativeisSobreAsFibras}.
\end{example}

\begin{example}
\label{exa:Pre-ordemT} \emph{(Total Sum Order)} One of the first
works to directly propose a regular method for ranking $\mathbb{T}$
is given in Al-Amleh \cite{Al-Amleh}. In \cite{Al-Amleh}, more precisely
in \emph{Definition 3.2}, the following total preorder on $\mathbb{T}$
is considered:
\[
\left(a_{1},a,a_{2}\right)\preceq_{T}\left(b_{1},b,b_{2}\right)\,\Longleftrightarrow\,\textrm{\ensuremath{a_{1}+a+a_{2}\leq b_{1}+b+b_{2}.}}
\]
It is straightforward to verify that $\preceq_{T}$ is regular. Note
that, for every $\alpha\in\mathbb{T}$,
\[
N_{ull}\left(\alpha\right)\subset\left\{ \beta\in\mathbb{T}\,:\,\beta\sim_{T}\alpha\right\} .
\]
 Based on \emph{Theorem} \emph{\ref{thm:OrdensCompativeisSobreAsFibras}},
it is simple to extend $\preceq_{T}$ to a regular order on $\mathbb{T}$:

Among the various extensions of $\preceq_{T}$, we define the \emph{total
sum order} as
\begin{equation}
\left(a_{1},a,a_{2}\right)\leq_{T}\left(b_{1},b,b_{2}\right)\,\Longleftrightarrow\,\left\{ \begin{array}{l}
\textrm{\ensuremath{a_{1}+a+a_{2}<b_{1}+b+b_{2},} or }\\
\textrm{\ensuremath{a_{1}+a+a_{2}=b_{1}+b+b_{2}} and \ensuremath{a<b}; or }\\
\textrm{\ensuremath{a_{1}+a+a_{2}=b_{1}+b+b_{2}} and \ensuremath{a=b} and \ensuremath{a_{2}\leq b_{2}}.}
\end{array}\right.\label{eq:OrdemSomaTotal}
\end{equation}
The total sum order, $\leq_{T}$, is one of the main orders presented
in this work because it is regular, satisfies the WLT, and has positive
$0$-symmetric numbers.

Not every regular extension of $\leq_{T}$ satisfies the WLT, as shown
by the order
\begin{equation}
\left(a_{1},a,a_{2}\right)\leq'_{T}\left(b_{1},b,b_{2}\right)\,\Longleftrightarrow\,\left\{ \begin{array}{l}
\textrm{\ensuremath{a_{1}+a+a_{2}<b_{1}+b+b_{2},} or }\\
\textrm{\ensuremath{a_{1}+a+a_{2}=b_{1}+b+b_{2}} and \ensuremath{a_{2}<b_{2}}; or }\\
\textrm{\ensuremath{a_{1}+a+a_{2}=b_{1}+b+b_{2}} and \ensuremath{a_{2}=b_{2}} and \ensuremath{a\leq b}.}
\end{array}\right.\label{eq:OrdemT'}
\end{equation}
This order is regular and has positive $0$-symmetric numbers, but
it does not satisfy the WLT: if $\alpha=\left(-9,1,8\right)$, then
$0<'_{T}\alpha,-\alpha$, for example.
\end{example}

Compatibility with \emph{MIN-MAX} operators is what we need to complete
\emph{Theorem} \emph{\ref{thm:BolaseIntervalos1}} and characterize
open and closed balls through intervals for the uncovered cases. For
this purpose, we will use the following auxiliary result:
\begin{lem}
\label{lem:DesigualdadeComOsMinimais} Let $\preceq$ be a regular
order satisfying WLT and possessing positive $0$-symmetric elements.
Given triangular fuzzy numbers $\beta=\left(b_{1},b,b_{2}\right)$
and $\gamma=\left(c_{1},c,c_{2}\right)$ with $\gamma\succ0$ and
$\gamma\notin I_{0}$:\\
i) if $b-b_{1}>c-c_{1}$ or $b_{2}-b>c_{2}-c$, then $\left|min\left(N_{ull}\left(\beta-\gamma\right)\right)-\beta\right|_{\preceq}=\beta-min\left(N_{ull}\left(\beta-\gamma\right)\right)\succ\gamma$;\\
ii) if $b-b_{1}>c_{2}-c$ or $b_{2}-b>c-c_{1}$, then $\left|min\left(N_{ull}\left(\beta+\gamma\right)\right)-\beta\right|_{\preceq}=min\left(N_{ull}\left(\beta+\gamma\right)\right)-\beta\succ\gamma$.
\end{lem}

\begin{proof}
\emph{Case (i):} suppose $b-b_{1}>c-c_{1}$ or $b_{2}-b>c_{2}-c$.
Let $\alpha_{0}=min\left(N_{ull}\left(\beta-\gamma\right)\right)$.
Then, by \emph{Remark \ref{obs:ElementosMinimaisAnuladores}},
\[
\begin{array}{rcl}
\alpha_{0} & = & min\left\{ \left(x,b-c,y\right)\in\mathbb{T}\,:\,x+y=b_{1}+b_{2}-c_{1}-c_{2}\right\} \\
 & = & \left\{ \begin{array}{l}
\textrm{\ensuremath{\left(b_{1}+b_{2}-c_{1}-c_{2}+c-b,b-c,b-c\right),} if \ensuremath{b_{1}+b_{2}-c_{1}-c_{2}}\ensuremath{\leq2b-2c}; or }\\
\textrm{\ensuremath{\left(b-c,b-c,b_{1}+b_{2}-c_{1}-c_{2}+c-b\right),} if \ensuremath{b_{1}+b_{2}-c_{1}-c_{2}}\ensuremath{\geq2b-2c}. }
\end{array}\right.
\end{array}
\]

By WLT, $-\gamma\prec0$. By compatibility with sum, $\beta-\gamma\prec\beta$.
Since $\alpha_{0}=min\left(N_{ull}\left(\beta-\gamma\right)\right)\preceq\beta-\gamma$,
the transitivity of the order implies that $\alpha_{0}\prec\beta$.
As $\alpha_{0}\notin N_{ull}\left(\beta\right)$, \emph{Proposition
\ref{prop:PropriedadeReciproca}.(i)} implies that 
\[
\left|\alpha_{0}-\beta\right|_{\preceq}=\beta-\alpha_{0}=\left\{ \begin{array}{l}
\textrm{\ensuremath{\left(b_{1}+c-b,c,c_{1}+c_{2}+b-c-b_{1}\right),} if \ensuremath{b_{1}+b_{2}-c_{1}-c_{2}}\ensuremath{\leq2b-2c}; or }\\
\textrm{\ensuremath{\left(c_{1}+c_{2}+b-c-b_{2},c,b_{2}+c-b\right),} if \ensuremath{b_{1}+b_{2}-c_{1}-c_{2}}\ensuremath{\geq2b-2c}. }
\end{array}\right.
\]

If $\ensuremath{b_{1}+b_{2}-c_{1}-c_{2}}\ensuremath{\leq2b-2c}$,
then we must have $b-b_{1}>c-c_{1}$; otherwise, by hypothesis, $b_{2}-b>c_{2}-c$,
and thus $2b-2c<b_{1}-c_{1}+b{}_{2}-c_{2}$, which is a contradiction.
Therefore $b-b_{1}>c-c_{1}$, and in this case, $\left(c_{1}+c_{2}+b-c-b_{1}\right)+\left(b_{1}+c-b\right)=c_{1}+c_{2}$
and $c_{1}+c_{2}+b-c-b_{1}>c_{1}+c_{2}+b_{'1}-c_{1}-b_{1}=c_{2}$.
By \emph{Theorem \ref{thm:OrdensCompativeisSobreAsFibras}}, $\gamma\prec\beta-\alpha_{0}=\left|\alpha_{0}-\beta\right|_{\preceq}$.

If $\ensuremath{b_{1}+b_{2}-c_{1}-c_{2}}\ensuremath{\geq2b-2c}$,
then we must have $b_{2}-b>c_{2}-c$; otherwise, $b-b_{1}$ must be
greater than $c-c_{1}$ (by hypothesis), leading to $\ensuremath{b_{1}+b_{2}-c_{1}-c_{2}}\ensuremath{<2b-2c}$,
a contradiction. Hence, $b_{2}+c-b>c_{2}-c+c=c_{2}$. Again, by\emph{
Theorem \ref{thm:OrdensCompativeisSobreAsFibras}}, $\gamma\prec\beta-\alpha_{0}=\left|\alpha_{0}-\beta\right|_{\preceq}$.

Case \emph{(ii)} follows the same line of reasoning.
\end{proof}

\begin{thm}
\label{thm:BolasEIntervalos2} Let $\preceq$ be a regular order satisfying
the WLT and having positive $0$-symmetric numbers. Given $\beta=\left(b_{1},b,b_{2}\right),\gamma=\left(c_{1},c,c_{2}\right)\in\mathbb{T}$
with $\gamma\succ0$ and $\gamma\notin I_{0}$, then:\\
i) if $b-b_{1}>c-c_{1}$ or $b_{2}-b>c_{2}-c$, as well as $b-b_{1}>c_{2}-c$
or $b_{2}-b>c-c_{1}$, then 
\[
\overline{\mathcal{B}}\left(\beta,\gamma\right)=\mathcal{B}\left(\beta,\gamma\right)=\left(\alpha_{1},\,\alpha_{2}\right)_{\preceq}\backslash N_{ull}\left(\alpha_{1}\right),
\]
where $\alpha_{1}=min\left(N_{ull}\left(\beta-\gamma\right)\right)$
and $\alpha_{2}=min\left(N_{ull}\left(\beta+\gamma\right)\right)$.\\
ii) if $b-b_{1}>c-c_{1}$ or $b_{2}-b>c_{2}-c$, but $b-b_{1}\leq c_{2}-c$
and $b_{2}-b\leq c-c_{1}$, then 
\[
\overline{\mathcal{B}}\left(\beta,\gamma\right)=\left[\alpha_{1},\,\alpha_{2}\right]_{\preceq}\backslash N_{ull}\left(\alpha_{1}\right),
\]
where $\alpha_{1}=min\left(N_{ull}\left(\beta-\gamma\right)\right)$
and $\alpha_{2}=\left(b_{2}+c_{1},b+c,b_{1}+c_{2}\right)$.\\
iii) if $b-b_{1}\leq c-c_{1}$ and $b_{2}-b\leq c_{2}-c$, but $b-b_{1}>c_{2}-c$
or $b_{2}-b>c-c_{1}$, then 
\[
\overline{\mathcal{B}}\left(\beta,\gamma\right)=\left[\alpha_{1},\,\alpha_{2}\right)_{\preceq}\backslash\left(\alpha_{1}+I_{0}\right),
\]
where $\alpha_{1}=\left(b_{2}-c_{2},b-c,b_{1}-c_{1}\right)$ and $\alpha_{2}=min\left(N_{ull}\left(\beta+\gamma\right)\right)$. 
\end{thm}

\begin{proof}
\emph{Case (i):} \emph{Lemma \ref{lem:DesigualdadeComOsMinimais}
}implies that $\alpha_{1},\alpha_{2}\notin\overline{\mathcal{B}}\left(\beta,\gamma\right)$.
Since $\gamma\succ0$ and $\gamma\notin I_{0}$, we have $-\gamma\prec0\prec\gamma$,
and by compatibility with sum, $\beta-\gamma\prec\beta\prec\beta+\gamma$.
By transitivity, $\alpha_{1}\prec\beta$. By \emph{Proposition \ref{prop:PropriedadeReciproca}.(ii),}
$\beta\prec\alpha_{2}$. 

If $\beta\preceq\alpha\prec\alpha_{2}$, then 
\[
\left|\alpha-\beta\right|_{\preceq}=\alpha-\beta\prec\alpha_{2}-\beta\preceq\gamma+\beta-\beta.
\]
Since $\alpha\notin N_{ull}\left(\gamma+\beta\right)$, $\alpha-\beta\notin N_{ull}\left(\gamma\right)$,
and by \emph{Proposition \ref{prop:PropriedadeReciproca}.(ii)}, we
conclude\emph{ $\alpha-\beta\prec\gamma$}. Thus, $\alpha\in\overline{\mathcal{B}}\left(\beta,\gamma\right)$. 

If $\alpha_{1}\prec\alpha\prec\beta$ and $\alpha\notin N_{ull}\left(\alpha_{1}\right)$,
then 
\[
\left|\alpha-\beta\right|_{\preceq}=\beta-\alpha\prec\beta-\alpha_{1}\prec\beta-\left(\beta-\gamma\right),
\]
 since $-\alpha\prec-\alpha_{1}$ by \emph{Proposition \ref{prop:PropriedadeReciproca}.(i)}.
As $\beta-\alpha\notin N_{ull}\left(\gamma\right)$, \emph{Proposition
\ref{prop:PropriedadeReciproca}.(ii) }implies $\beta-\alpha\prec\gamma$.
Hence, $\alpha\in\overline{\mathcal{B}}\left(\beta,\gamma\right)$.

If $\alpha\in N_{ull}\left(\alpha_{1}\right)$, then there exists
$\delta\in N_{ull}\left(0\right)$ such that $\alpha=\alpha_{1}+\delta$.
In this case, 
\[
\left|\alpha-\beta\right|_{\preceq}=\beta-\alpha=\beta-\alpha_{1}+\delta\succ\gamma+\delta,
\]
 since $\beta-\alpha_{1}\succ\gamma$ \emph{by Lemma \ref{lem:DesigualdadeComOsMinimais}}.
Thus, $N_{ull}\left(\alpha_{1}\right)\cap\overline{\mathcal{B}}\left(\beta,\gamma\right)=\emptyset$. 

For $\alpha\prec\alpha_{1}$ or $\alpha\succ\alpha_{2}$: 

$\bullet$ if $\alpha\succ\alpha_{2}$, then $\left|\alpha-\beta\right|_{\preceq}=\alpha-\beta\succ\alpha_{2}-\beta\succ\gamma$. 

$\bullet$ if $\alpha\prec\alpha_{1}$, then $\alpha\notin N_{ull}\left(\alpha_{1}\right)$,
and $-\alpha_{1}\prec-\alpha$ by \emph{Proposition \ref{prop:PropriedadeReciproca}.(i)}.
As consequence of \emph{Lemma \ref{lem:DesigualdadeComOsMinimais}},
$\left|\alpha-\beta\right|_{\preceq}=\beta-\alpha\succ\beta-\alpha_{1}\succ\gamma$.

Therefore,
\[
\overline{\mathcal{B}}\left(\beta,\gamma\right)=\mathcal{B}\left(\beta,\gamma\right)=\left(\alpha_{1},\,\alpha_{2}\right)_{\preceq}\backslash N_{ull}\left(\alpha_{1}\right).
\]

\emph{Case (ii)}: For $\beta\preceq\alpha\preceq\alpha_{2}$, we apply
the arguments established in \emph{Theorem \ref{thm:BolaseIntervalos1}.(ii)}.
For $\alpha_{1}\preceq\alpha\prec\beta$, we employ the methodology
developed in \emph{Case (i)} above. \emph{(iii)} follows through inversion
of these approaches.
\end{proof}

\section{\label{sec:CompatibilidadeProjNatural} Compatibility with natural
projection and the upper sum order}

If the set of triangular fuzzy numbers is equipped with a regular
order satisfying the weak law of trichotomy, \emph{Theorem} \emph{\ref{thm:OrdensCompativeisSobreAsFibras}}
characterizes the only two possible ways to rank each fiber $\pi^{-1}\left(t\right)$,
$t\in\mathbb{R}$. To completely determine this order on $\mathbb{T}$,
we then need to understand how the order compares elements from different
fibers.

The simplest way to \textquotedbl accommodate\textquotedbl{} the
fibers of the natural projection in $\mathbb{T}$ is to require that
the order does not \textquotedbl mix\textquotedbl{} distinct fibers.
In this case, the order is compatible with the natural projection
(\emph{Definition} \emph{\ref{def:CompProjNatural}}). This is consistent
with the notion that a triangular fuzzy number $\alpha=\left(a_{1},a,a_{2}\right)\in\mathbb{T}$
is a special case of a fuzzy quantity that aims to abstract the idea
of a \textquotedbl number close to $a$\textquotedbl , where $a$
is the \textquotedbl expected/likely value\textquotedbl{} and $a_{2}-a$,
$a-a_{1}$ represent the \textquotedbl upper and lower uncertainty
margins\textquotedbl , respectively. Ranking methods compatible with
the natural projection first compare the most probable values of each
triangular fuzzy number and only use the uncertainty margins to decide
in case of equality.
\begin{lem}
\label{lem:CaracOrdensPreservadasPelaProjNatural} An order $\preceq$
is compatible with the natural projection if and only if
\[
\left(x_{1},a,y_{1}\right)\preceq\left(x_{2},b,y_{2}\right)\Longleftrightarrow\begin{cases}
\textrm{\ensuremath{a<b,} or }\\
\textrm{\ensuremath{a=b} and \ensuremath{\left(x_{1},y_{1}\right)\leq_{a}\left(x_{2},y_{2}\right)},} & \textrm{}
\end{cases}
\]
for some order $\leq_{a}$ on $\left(-\infty,a\right]\times\left[a,\infty\right)\subseteq\mathbb{R}^{2}$.
\end{lem}

\begin{proof}
Suppose that $\preceq$ is compatible with the natural projection.
Take $\alpha=\left(x_{1},a,y_{1}\right)$ and $\beta=\left(x_{2},b,y_{2}\right)$.
If $b>a$, then $\alpha\preceq\beta$ and $\alpha\neq\beta$. Hence,
$\alpha\prec\beta$.

For a fixed $a\in\mathbb{R}$, we define the order $\leq_{a}$ on
$\left(-\infty,a\right]\times\left[a,\infty\right)$ by pulling back
$\preceq$ from $\pi^{-1}\left(a\right)$:
\[
\left(x_{1},y_{1}\right)\leq_{a}\left(x_{2},y_{2}\right)\Longleftrightarrow\left(x_{1},a,y_{1}\right)\preceq\left(x_{2},a,y_{2}\right).
\]
 Having done this for every scalar a a, we have
\[
\left(x_{1},a,y_{1}\right)\preceq\left(x_{2},b,y_{2}\right)\Longleftrightarrow\begin{cases}
\textrm{\ensuremath{a<b,} or }\\
\textrm{\ensuremath{a=b} and \ensuremath{\left(x_{1},y_{1}\right)\leq_{a}\left(x_{2},y_{2}\right)},} & \textrm{}
\end{cases}
\]
for $\left(x_{1},a,y_{1}\right),\left(x_{2},b,y_{2}\right)\in\mathbb{T}$.

The converse of this result is trivial.
\end{proof}
\begin{rem}
If $\preceq$ is a total order compatible with the natural projection
$\pi$, then 
\[
\left\{ \alpha\in\mathbb{T}\,:\,\pi\left(\alpha\right)>0\right\} \subseteq P_{\preceq}.
\]
 If $\preceq_{1},\preceq_{2}$ are two orders compatible with $\pi$,
then 
\[
P_{\preceq_{1}}\triangle P_{\preceq_{2}}=\left\{ \alpha\in\pi^{-1}\left(0\right)\,:\,0\prec_{1}\alpha\right\} \triangle\left\{ \alpha\in\pi^{-1}\left(0\right)\,:\,0\prec_{2}\alpha\right\} ,
\]
where $\triangle$ denotes the symmetric difference. That is, for
distinct orders compatible with the natural projection, the sets of
positive elements differ only on the fiber $\pi^{-1}\left(0\right)$.
\end{rem}

\begin{prop}
\label{prop:PositivosNumaOrdemNormal} If $\preceq$ is a regular
order compatible with the natural projection and admitting positive
$0$-symmetric elements, then 
\[
\left[\pi^{-1}\left(\left(0,\infty\right)\right)\cup\left\{ \left(x,0,y\right)\in\pi^{-1}\left(0\right)\,:\,\textrm{\ensuremath{-x\leq y} e \ensuremath{y>0}}\right\} \right]\subseteq P_{\preceq}\subseteq P_{\leq_{231}},
\]
where $\leq_{231}$ is defined in Example \ref{exa:Ordensijk}. Furthermore,
$\preceq$ satisfies the WLT if and only if 
\[
P_{\preceq}=\pi^{-1}\left(\left(0,\infty\right)\right)\cup\left\{ \left(x,0,y\right)\in\pi^{-1}\left(0\right)\,:\,\textrm{\ensuremath{-x\leq y} e \ensuremath{y>0}}\right\} .
\]
\end{prop}

\begin{proof}
By \emph{Lemma \ref{lem:CaracOrdensPreservadasPelaProjNatural}},
$\pi^{-1}\left(\left(0,\infty\right)\right)\subset P_{\preceq}$.
Since $\preceq$ is compatible with \emph{MIN-MAX} and $I_{0}\subset P_{\preceq}$,
for any $x<0$, $\left(x,0,y\right)\in P_{\preceq}$ for all $y\geq-x$,
because $0\prec\left(x,0,-x\right)\preceq\left(x,0,y\right)$. 

Furthermore, if $\preceq$ satisfies the WLT, then:
\[
\left[\pi^{-1}\left(\left(-\infty,0\right)\right)\cup\left\{ \left(x,0,y\right)\in\pi^{-1}\left(0\right)\,:\,\textrm{\ensuremath{-x>y} }\right\} \right]\cap P_{\preceq}=\emptyset,
\]
since 
\[
\left[\pi^{-1}\left(\left(0,\infty\right)\right)\cup\left\{ \left(x,0,y\right)\in\pi^{-1}\left(0\right)\,:\,\textrm{\ensuremath{-x\leq y} e \ensuremath{y>0}}\right\} \right]\subseteq P_{\preceq}.
\]
\end{proof}
\emph{Proposition} \emph{\ref{prop:PositivosNumaOrdemNormal}} reveals
the order $\leq_{231}$, defined in \emph{Example} \emph{\ref{exa:Ordensijk}},
has the largest set of positive elements among all regular orders
compatible with the natural projection that admit positive $0$-symmetric
elements. Among these, the orders satisfying the WLT are those with
the smallest set of positive elements.

As a consequence of \emph{Theorem} \emph{\ref{thm:OrdensCompativeisSobreAsFibras}}
and \emph{Lemma} \emph{\ref{lem:CaracOrdensPreservadasPelaProjNatural}},
there are exactly two regular orders compatible with the natural projection
that satisfy the WLT:
\begin{itemize}
\item The \emph{upper-sum} order, which admits positive $0$-symmetric elements:
\begin{equation}
\left(a_{1},a,a_{2}\right)\leq^{+}\left(b_{1},b,b_{2}\right)\Longleftrightarrow\left\{ \begin{array}{l}
\textrm{\ensuremath{a<b}, or }\\
\textrm{\ensuremath{a=b} and \ensuremath{a_{1}+a_{2}<b_{1}+b_{2},} or }\\
\textrm{\ensuremath{a=b} and \ensuremath{a_{1}+a_{2}=b_{1}+b_{2}} and \ensuremath{a_{2}\leq b_{2}}}.
\end{array}\right.\label{eq:SomaSuperior}
\end{equation}
\item The l\emph{ower-sum} order, which admits negative $0$-symmetric elements:\emph{
}
\begin{equation}
\left(a_{1},a,a_{2}\right)\leq_{+}\left(b_{1},b,b_{2}\right)\Longleftrightarrow\left\{ \begin{array}{l}
\textrm{\ensuremath{a<b}, or }\\
\textrm{\ensuremath{a=b} and \ensuremath{a_{1}+a_{2}<b_{1}+b_{2},} or }\\
\textrm{\ensuremath{a=b} and \ensuremath{a_{1}+a_{2}=b_{1}+b_{2}} and \ensuremath{a_{1}\leq b_{1}}.}
\end{array}\right.\label{eq:SomaInferior}
\end{equation}
 
\end{itemize}
The lower sum order has no positive $0$-symmetric elements ($I_{0}\cap P_{\leq_{+}}=\emptyset$)
and thus does not induce a fuzzy absolute value. Moreover, its set
of positive elements is contained in the set of positive elements
of every regular order compatible with the natural projection. 

Thus, the following result is established:
\begin{thm}
\label{thm:UnicidadeOrdemNormalLFT} The upper-sum and lower-sum orders
are the only total orders on $\mathbb{T}$ that are regular, compatible
with the natural projection, and satisfy WLT. 
\end{thm}

Compatibility with the natural projection is the most restrictive
of the compatibility conditions we are considering, as shown in \emph{Lemma}
\emph{\ref{lem:CaracOrdensPreservadasPelaProjNatural}}. The following
example shows that this condition can arise naturally when considering,
for instance, methods for ranking fuzzy numbers based on the concept
of \emph{weighted possibilistic mean} introduced by Fullér and Majlender
\cite{Fuller and Majlender}.
\begin{example}
\label{exa:OrdemMolinari} Based on the concept of weighted possibilistic
mean, Molinari \cite{Molinari} considers a new ranking criterion
for generalized triangular fuzzy numbers ---a set of fuzzy numbers
that strictly contains $\mathbb{T}$. First, Molinari introduces a
partial order relation on generalized triangular fuzzy numbers which,
when restricted to $\mathbb{T}$, gives us the following partial order
(see \emph{Theorem 6.1} of \cite{Molinari}): 
\[
\left(a_{1},a,a_{2}\right)\preceq^{w}\left(b_{1},b,b_{2}\right)\,\Longleftrightarrow\,\textrm{\ensuremath{a\leq b} and \ensuremath{a_{1}+2a+a_{2}\leq b_{1}+2b+b_{2}}},
\]
 which is compatible with arithmetic operations and \emph{MIN-MAX}
operators. 

Subsequently, Molinari obtains a total preorder on generalized triangular
fuzzy numbers that, when restricted to $\mathbb{T}$, yields (as a
consequence of \emph{Theorem 6.1} and \emph{Proposition 6.2} of \cite{Molinari})
the following total preorder:
\[
\left(a_{1},a,a_{2}\right)\preceq^{W}\left(b_{1},b,b_{2}\right)\Longleftrightarrow\left\{ \begin{array}{l}
\textrm{\ensuremath{a<b}, or }\\
\textrm{\ensuremath{a=b} and \ensuremath{a_{1}+a_{2}\leq b_{1}+b_{2}}}.
\end{array}\right.
\]
We have that $\preceq^{W}$ is a total preorder, regular, and compatible
with the natural projection. Note that for any $\alpha\in\mathbb{T}$,
\[
\left\{ \beta\in\mathbb{T}\,:\,\beta\sim^{W}\alpha\right\} =N_{ull}\left(\alpha\right).
\]

Because it's compatible with arithmetic operations, \emph{Theorem}
\emph{\ref{th:OrdemNosAnuladores}} guarantees only two possible orderings
for nullifying sets. Therefore, there exist exactly two extensions
of $\preceq^{W}$ to a total order on $\mathbb{T}$: the upper-sum
order $\leq^{+}$ and the lower-sum order $\leq_{+}$. In particular,
the upper sum order is the unique extension of $\preceq^{W}$ that
induces a fuzzy absolute value.
\end{example}

\begin{thm}
\label{thm:OrdensRegularesCompativeisProjNatural} All regular orders
compatible with the natural projection and admitting positive 0-symmetric
numbers induce the same fuzzy absolute value, $\left|\cdot\right|_{\leq^{+}}$.
\end{thm}

\begin{proof}
Let \ensuremath{\preceq} be a regular order compatible with the natural
projection that admits positive 0-symmetric elements. We show that
$\left|\cdot\right|_{\preceq}=\left|\cdot\right|_{\leq^{+}}$.

By compatibility with $\pi$:

$\bullet$ if $\pi\left(\alpha\right)>0$, then $\left|\alpha\right|_{\preceq}=\alpha=\left|\alpha\right|_{\leq^{+}}$
because $-\alpha\prec0\prec\alpha$.

$\bullet$ if $\pi\left(\alpha\right)<0$, then $\left|\alpha\right|_{\preceq}=-\alpha=\left|\alpha\right|_{\leq^{+}}$.

For the case $\pi\left(\alpha\right)=0$:

$\bullet$ $MIN\left(\alpha,-\alpha\right)\in\left\{ \alpha,-\alpha\right\} $;

$\bullet$ $MAX\left(\alpha,-\alpha\right)=\alpha$ if and only if
$a_{2}\geq-a_{1}$, where $\alpha=\left(a_{1},0,a_{2}\right)$;

Equivalently, $\alpha\geq_{KY}-\alpha$ if and only if $\alpha\in P_{\leq^{+}}\cup\left\{ 0\right\} $.
Thus, $\left|\alpha\right|_{\preceq}=\left|\alpha\right|_{\leq^{+}}$
for all $\alpha\in\mathbb{T}$.
\end{proof}
\begin{example}
\label{exa:Ordem231InduzValorAbsolutoDifuso} By the previous theorem,
$\left|\cdot\right|_{\leq_{231}}=\left|\cdot\right|_{\leq^{+}}$,
where $\leq_{231}$ is the order defined in \emph{Example} \emph{\ref{exa:Ordensijk}}.
For arbitrary triangular fuzzy numbers, it is straightforward to verify
that:\\
\emph{i) }$0\leq_{231}\left|\alpha\right|_{\leq_{231}}$, and $0=\left|\alpha\right|_{\leq_{231}}$
if and only if $\alpha=0$; \\
\emph{ii)} $\left|t\alpha\right|_{\leq_{231}}=\left|t\right|\left|\alpha\right|_{\leq_{231}}$,
for $t\in\mathbb{R}$;\\
\emph{iii)} $\left|\alpha+\beta\right|_{\leq_{231}}\leq_{231}\left|\alpha\right|_{\leq_{231}}+\left|\beta\right|_{\leq_{231}}$. 

We now verify that:\\
\emph{iv) (triangular inequality) }$\left|\alpha-\gamma\right|_{\leq_{231}}\leq_{231}\left|\alpha-\beta\right|_{\leq_{231}}+\left|\beta-\gamma\right|_{\leq_{231}}$:

Suppose $\left|\cdot\right|_{\leq_{231}}$ does not satisfy the triangular
inequality. Then there exist $\alpha=\left(a_{1},a,a_{2}\right)$,
$\beta=\left(b_{1},b,b_{2}\right)$ and $\gamma=\left(c_{1},c,c_{2}\right)$
such that
\[
\left|\alpha-\gamma\right|_{\leq_{231}}>_{231}\left|\alpha-\beta\right|_{\leq_{231}}+\left|\beta-\gamma\right|_{\leq_{231}}.
\]
 Since $\left|\cdot\right|_{\leq_{231}}=\left|\cdot\right|_{\leq^{+}}$,
by \emph{Theorem \ref{thm:OrdensRegularesCompativeisProjNatural}},
we have 
\[
\left|\alpha-\gamma\right|_{\leq_{231}}>_{231}\left|\alpha-\beta\right|_{\leq_{231}}+\left|\beta-\gamma\right|_{\leq_{231}}\geq^{+}\left|\alpha-\gamma\right|_{\leq_{231}}.
\]
 Without loss of generality, assume $\left|\alpha-\gamma\right|_{\leq_{231}}=\alpha-\gamma=\left(a_{1}-c_{2},a-c,a_{2}-c_{1}\right)$.
Let $\eta=\left|\alpha-\beta\right|_{\leq_{231}}+\left|\beta-\gamma\right|_{\leq_{231}}$.
The inequality above implies: 
\[
\textrm{\ensuremath{\pi\left(\eta\right)=a-c}, \ensuremath{a_{2}-c_{1}>sup\,S_{upp}\left(\eta\right)} and \ensuremath{inf\,S_{upp}\left(\eta\right)+sup\,S_{upp}\left(\eta\right)\geq a_{1}+a_{2}-c_{1}-c_{2}}.}
\]
 There are only four possibilities for $\eta$:
\[
\eta=\left\{ \begin{array}{l}
\textrm{\ensuremath{\left(a_{1}-b_{2}+b_{1}-c_{2},a-c,a_{2}+b_{2}-b_{1}-c_{1}\right),} or}\\
\textrm{\ensuremath{\left(a_{1}+c_{1}-2b_{2},a+c-2b,a_{2}+c_{2}-2b_{1}\right),} or}\\
\textrm{\ensuremath{\left(2b_{1}-a_{2}-c_{2},2b-a-c,2b_{2}-a_{1}-c_{1}\right),} or}\\
\textrm{\ensuremath{\left(b_{1}+c_{1}-a_{2}-b_{2},c-a,b_{2}+c_{2}-a_{1}-b_{1}\right).} }
\end{array}\right.
\]
 For any of these cases, the conditions $\pi\left(\eta\right)=a-c$,
$a_{2}-c_{1}>sup\,S_{upp}\left(\eta\right)$ and $inf\,S_{upp}\left(\eta\right)+sup\,S_{upp}\left(\eta\right)\geq a_{1}+a_{2}-c_{1}-c_{2}$
cannot hold simultaneously. Therefore, the triangular inequality holds.

Hence, even without satisfying WLT, the order $\leq_{231}$ induces
a notion of absolute value that allows us to define a fuzzy distance
induced by $\leq_{231}$.
\end{example}

\section{\label{sec:ExemplosNumericos} Numerical examples}

Triangular fuzzy numbers are, by far, the most prevalent choice for
illustrative examples, particularly those demonstrating compatibility/incompatibility
between different ranking methods for fuzzy numbers. In references
\cite{Akyar et al.} and \cite{FRM}, the authors compiled comprehensive
sets of ranking methodologies applied to critical decision-making
scenarios, enabling comparative analysis with their respective proposed
approaches. We utilize these established examples to evaluate and
contrast the results with the ordering relations introduced in the
present work. To facilitate this comparative study, we adopt a systematic
nomenclature for the methods enumerated in each of these publications.

From \emph{Table I }of \cite{FRM}, we employ the following notation:

$\bullet$ Yager's methods: $\preceq_{Y_{1}}$, $\preceq_{Y_{2}}$,
and $\preceq_{Y_{3}}$;

$\bullet$ The Baas-Kwakernaak method: $\preceq_{BK}$;

$\bullet$ Baldwin-Guild methods: $\preceq_{BG_{1}}$, $\preceq_{BG_{2}}$,
and $\preceq_{BG_{3}}$;

$\bullet$ Jain's method: $\preceq_{J}$;

$\bullet$ Dubois-Prade methods: $\preceq_{DP_{1}}$ through $\preceq_{DP_{4}}$;

$\bullet$ The authors' own methods: $\preceq_{FGM_{1}}$ to $\preceq_{FGM_{5}}$,
including both optimistic and pessimistic approaches (see \emph{Example}
\emph{\ref{exa:OrdensPessimistaeOtimista}})

Following the enumeration of ranking methods in \emph{Table 2} of
\cite{Akyar et al.}, we additionally define:

$\bullet$ An alternative Yager method: $\preceq_{Y_{4}}$;

$\bullet$ Murakami et al.'s method: $\preceq_{MMI}$;

$\bullet$ Cheng's method: $\preceq_{C}$;

$\bullet$ Chen-Chen method: $\preceq_{CC}$;

$\bullet$ Lee-Chen method: $\preceq_{LC}$;

$\bullet$ The novel method proposed in \cite{Akyar et al.}: $\preceq_{AAD}$.

This systematic classification enables rigorous comparative analysis
between existing ranking methodologies and the novel ordering relations
proposed in our current work. 
\begin{example}
\label{exa:1} Consider the following triangular fuzzy numbers: $\alpha=\left(-0.5,-0.3,-0.1\right)$,
$\beta=\left(0.2806,0.4806,0.6806\right)$ and $\gamma=0.7=\left(0.7,0.7,0.7\right)$.
If $\preceq$ is a reasonable ranking method over some set of fuzzy
numbers containing $\mathbb{T}$, then we must have $\alpha\prec-\alpha$
and $\beta\prec\gamma$ according to \emph{Definition} \emph{\ref{def:MetodoRazoavel}.(vi)}.

As can be observed in \emph{Section 4} of \cite{Akyar et al.}, Cheng's
method yields $\alpha\succ_{C}-\alpha$, differing from all other
analyzed methods. Among these methods, only the $\preceq_{CC}$ and
$\preceq_{LC}$ methods are capable of comparing $\beta$ and $\gamma$,
but $\beta\succ_{LC}\gamma$. Consequently, no reasonable ranking
method can be derived from Cheng's method $\preceq_{C}$ or Lee-Chen's
method $\preceq_{LC}$.

The remaining methods listed in \cite{Akyar et al.} agree with both
the total sum and upper sum orders: $\alpha\prec-\alpha$ and $\beta\prec\gamma$
for $\prec\in\left\{ <_{T},<^{+}\right\} $. 
\end{example}

\begin{example}
\label{exa:2} In general, ranking methods exhibit significant difficulty
in selecting between pairs of triangular fuzzy numbers belonging to
the same nullifying set. This becomes evident when six of the nineteen
methods described in \cite{FRM} fail to distinguish between the numbers
$\alpha=\left(0.2,0.5,0.8\right)$ and $\beta=\left(0.4,0.5,0.6\right)$:
these fuzzy numbers are considered equivalent by the methods $\preceq_{Y_{1}}$,
$\preceq_{Y_{3}}$, $\preceq_{BK}$, $\preceq_{DP_{1}}$, $\preceq_{DP_{4}}$
and $\preceq_{FGM_{4}}$.

Seven of the nineteen methods demonstrate agreement with both the
total sum and upper sum orders, showing preference for $\alpha$ over
$\beta$: $\preceq_{Y_{2}}$, $\preceq_{BG_{2}}$, $\preceq_{J}$,
$\preceq_{DP_{2}}$, $\preceq_{FGM_{2}}$, $\preceq_{FGM_{5}}$, and
the optimistic method. The remaining methods prefer $\beta$ over
$\alpha$.

Consequently, no fuzzy absolute value can be induced by any order
on $\mathbb{T}$ that extends the following methods: $\preceq_{BG_{1}}$,
$\preceq_{BG_{3}}$, $\preceq_{DP_{3}}$, $\preceq_{FGM_{1}}$, $\preceq_{FGM_{3}}$,
or the pessimistic method.
\end{example}

\begin{example}
\label{exa:3} The method $\preceq_{AAD}$ likewise does not admit
any extension that induces a fuzzy absolute value, since it ranks
$\left(0.1,0.3,0.5\right)$ as being smaller than $\left(0.2,0.3,0.4\right)$.
The fact that half of the methods listed in \cite{Akyar et al.} agree
with $\preceq_{AAD}$ and show preference for $\left(0.2,0.3,0.4\right)$,
while the other half prefer $\left(0.1,0.3,0.5\right)$, serves as
further evidence of how critical nullifying sets are when the subject
is ordering.
\end{example}

\begin{example}
\label{exa:4} Another critical case arises in the comparison between
$\alpha=\left(0.35,0.5,1\right)$ and $\beta=\left(0.15,0.65,0.8\right)$.
Among the 19 methods analyzed in \cite{FRM}, five prefer $\beta$
over $\alpha$, agreeing with the upper sum order $\alpha<^{+}\beta$.
These are: $\preceq_{Y_{2}}$, $\preceq_{BK}$, $\preceq_{J}$, $\preceq_{DP_{1}}$,
and $\preceq_{D_{4}}$. The remaining 14 methods prefer $\alpha$
over $\beta$, and agree with the total sum order $\alpha>_{T}\beta$.
\end{example}

\begin{description}
\item [{Acknowledgments}] I am particularly grateful to Nicolas Zumelzu
for introducing me to fuzzy number theory through both personal communications
and his published work, and to B. Bedregal for reading and providing
comments on the initial manuscript.
\end{description}

\end{document}